\documentclass[11pt]{amsart}
\usepackage[margin=1.1in]{geometry}
\usepackage{amsmath,amssymb,amsthm}
\usepackage{amscd}
\usepackage{needspace}
\usepackage{xcolor}
\usepackage[colorlinks=true,linkcolor=blue!60!black,citecolor=blue!60!black]{hyperref}

\newtheorem{theorem}{Theorem}[section]
\newtheorem*{thmArep}{Theorem~A}
\newtheorem*{thmBrep}{Theorem~B}
\newtheorem*{thmCrep}{Theorem~C}
\newtheorem*{thmDrep}{Theorem~D}
\newtheorem*{conjrep}{Conjecture}
\newtheorem{proposition}[theorem]{Proposition}
\newtheorem{lemma}[theorem]{Lemma}
\newtheorem{corollary}[theorem]{Corollary}
\newtheorem{conjecture}[theorem]{Conjecture}
\theoremstyle{definition}
\newtheorem{definition}[theorem]{Definition}
\newtheorem{remark}[theorem]{Remark}
\newtheorem{example}[theorem]{Example}

\newcommand{\Z}{\mathbb{Z}}
\newcommand{\Q}{\mathbb{Q}}
\newcommand{\F}{\mathbb{F}}
\newcommand{\HH}{\mathbb{H}}
\newcommand{\id}{\mathrm{id}}
\newcommand{\Sp}{\mathrm{Sp}}
\newcommand{\SU}{\mathrm{SU}}
\newcommand{\U}{\mathrm{U}}

\newcommand{\loops}{\Omega}

\title[BV rigidity in the string topology of compact Lie groups]{BV
rigidity in the string topology of compact Lie groups}

\author{Mukilraj K}
\address{Department of Mathematics and Statistics, Indian
Institute of Science Education and Research Kolkata, Mohanpur,
West Bengal, India - 741 246}
\email{mk21rs045@iiserkol.ac.in}


\begin{document}

\begin{abstract}
Let $G$ be a connected compact Lie group and $R$ a commutative ring
over which $H_*(G;R)$ is an exterior algebra on odd primitive
generators. We prove a rigidity theorem: the Batalin-Vilkovisky (BV)
algebra $\HH_*(LG;R)$ of Chas and Sullivan on the homology of the
free loop space is that of the product of odd spheres with the
classical exponents of $G$, tensored with the group ring of the
torsion of $\pi_1(G)$. Conversely, over an integral
domain the graded algebra $\HH_*(LG;R)$ already determines the
exponents and the group ring, hence the BV algebra. The isomorphism
with the product of spheres is algebraic: for $\Sp(n)$, $n\geq2$,
and $\SU(n)$, $n\geq3$, no map of spaces induces it over $\Z$. If
moreover $G$ is simply connected, then $\HH_*(LG;R)$, as a BV
algebra, is the Hochschild cohomology of $H^*(G;R)$ with the BV
operator induced by Poincar\'e duality.
\end{abstract}

\maketitle
\tableofcontents

\section{Introduction and statement of results}\label{sec:intro}

Let $M$ be a based closed oriented manifold of dimension $d$, let
$S^1=\mathbb{R}/\Z$ denote the circle, and let
$LM=\mathrm{Map}(S^1,M)$ be the free loop space, with
$\loops M\subset LM$ denoting the loops mapping $0$ to the
basepoint. In their seminal paper \cite{CS}, Chas and Sullivan
showed that the shifted homology $\HH_*(LM)=H_{*+d}(LM)$, called
the \emph{loop homology} of $M$, is a Batalin-Vilkovisky (BV)
algebra: the loop product combines the intersection of cycles
inside $M$ with the concatenation of loops, and the BV operator
$\Delta$ of degree one is induced by the rotation of loops.

For any Lie group $G$, right translation decomposes the free loop
space as $LG\cong\loops G\times G$. For compact $G$, Hepworth
\cite{Hepworth} uses this decomposition to determine the BV
algebra $\HH_*(LG)$ for every compact Lie group. The loop product
is the product of the Pontryagin ring $H_*(\loops G)$ with the
intersection ring $\HH_*(G)$,
\[
\HH_*(LG)\;\cong\;H_*(\loops G)\otimes\HH_*(G)
\]
as graded algebras, and the BV operator $\Delta$ is given in terms
of the coproduct of $H_*(\loops G)$, the homology suspension and
the Pontryagin action (Theorem~\ref{thm:hepworth1}). When
$H_*(G)$ is a Hopf algebra, choose a basis $p_1,\dots,p_n$ of its
odd primitives and the dual classes $p^i\in H^*(G)$. Then
\[
\Delta\;=\;\sum_{i=1}^{n}\partial_i\otimes\delta_i,
\qquad
\partial_i(a)=\sum\big\langle p^i,\sigma(a_{(1)})\big\rangle\,a_{(2)},
\qquad
\delta_i(x)=p_i\,x,
\]
where $\psi(a)=\sum a_{(1)}\otimes a_{(2)}$ is the coproduct of
$H_*(\loops G)$, $\sigma\colon H_*(\loops G)\to H_{*+1}(G)$ is the
homology suspension, and $\delta_i$ is the Pontryagin action of
$H_*(G)$ on $\HH_*(G)$; both $\partial_i$ and $\delta_i$ are
derivations (Proposition~\ref{prop:hepworth3}; see
Section~\ref{sec:hepworth} for details). For a large class of
Lie groups, $H_*(\loops G)$ is a polynomial ring on generators
$b_1,\dots,b_n$ which the homology suspension maps bijectively
onto the basis of odd primitives, $\sigma(b_k)=p_k$. With this
assumption the derivation
$\partial_k$ is determined by its values on the generators,
\[
\partial_k(b_l)\;=\;\delta_{kl}+\gamma_{kl},
\qquad
\gamma_{kl}\in H_{|b_l|-|b_k|}(\loops G).
\]
We call the $\gamma_{kl}$ the \emph{mixing terms} of the
presentation (Definition~\ref{def:mixing}). They are the terms of
$\Delta$ pairing generators of different degrees, $b_l$ with $p_k$
for $|b_l|>|b_k|$, and $\Delta$ is a pure tensor exactly when they
all vanish, as they do for primitive generators, where
$\psi(b_l)=b_l\otimes1+1\otimes b_l$ gives
$\partial_k(b_l)=\langle p^k,\sigma(b_l)\rangle=\delta_{kl}$.

Hepworth computed $\HH_*(L\,\mathrm{SO}(n))$ over $\Q$ and $\F_2$
\cite[Theorems~4 and~5]{Hepworth}: over $\Q$, $\partial_i$ sends
the generator $\alpha_j$ to $\alpha_{j-2i+1}$ for $j\geq2i-1$.
Tamanoi computed $\HH_*(L\SU(n{+}1);\Z)$ by a different method
\cite[Theorem~A]{Tamanoi}: in Hepworth's form,
$\Delta=\sum_{\ell}D_{2\ell}\otimes\partial/\partial\alpha_{2\ell+1}$
with $D_{2\ell}=\sum_{k=\ell}^{n}e_{2k-2\ell}\,\partial/\partial e_{2k}$
and $e_0=1$, so $e_{2k}$ is sent to $e_{2k-2\ell}$ for $k>\ell$.
In his thesis, Kupers applied Hepworth's presentation to $\SU(n)$
and $\Sp(n)$ over $\Z$, to $G_2$ over $\Q$ and $\F_2$, and to
$F_4$ over $\Q$ \cite[Theorems~3.1, 4.1, 5.1 and~6.1]{Kupers}. For
$\Sp(n)$ the $k$-th summand of $\Delta$ sends
$z_{2l-1}\otimes\alpha_{4k-1}$ to a multiple of $z_{2(l-k)}$ for
$k<l$ (Remark~\ref{rem:kupers}), and for $G_2$ over $\Q$ it sends
$\sigma_{10}\otimes l_{-3}$ to $3\sigma_2^4$. Every one of these
presentations has nontrivial mixing terms.

Over $\Q$ the mixing is an artifact. For $\SU(n{+}1)$, Tamanoi
passes to primitive generators
$h_{2\ell}=e_{2\ell}+(\text{decomposables})$ of
$H_*(\loops\SU(n{+}1);\Q)$ and the BV algebra splits as
$\bigotimes_{\ell=1}^{n}\HH_*(LS^{2\ell+1};\Q)$, the loop homology
of the product of odd spheres $S^3\times S^5\times\cdots\times
S^{2n+1}$, which has the rational homotopy type of $\SU(n{+}1)$
\cite[Section~4 and Corollary~C]{Tamanoi}. For every simply
connected $G$ the same follows from formality and the theorem of
F\'elix and Thomas \cite{FelixThomas}: $\HH_*(LG;\Q)$ is the loop
homology of the product of odd spheres with the rational homotopy
type of $G$ (Remark~\ref{rem:rational}). Over $\Z$ the primitives
do not generate $H_*(\loops G)$: for $\SU(3)$ the
primitive in degree $4$ is $2e_4-e_2^2$, by
\cite[Theorem~2-1]{Tamanoi}, and for $\Sp(2)$ the primitive in
degree $6$ is $3z_3-2z_1^3$ (Remark~\ref{rem:degen}). We would
like to ask whether the mixing over $\Z$ is real, that is,
whether the integral BV algebra distinguishes $G$ from the product
of odd spheres with the same rational homotopy type, or from other
Lie groups.

For a large family of groups the mixing vanishes. We
call $(G,R)$ an \emph{exterior pair} if $H_*(G;R)$ is an exterior
algebra generated by primitive elements of odd degrees
$2m_1+1,\dots,2m_r+1$; the $m_i$ are then the exponents of $G$
(Theorem~\ref{thm:whichpairs}). For example, $(G,\Q)$ is always an
exterior pair, and $(G,\Z)$ is one precisely when $G$ is
diffeomorphic to a torus times a product of groups $\SU(n)$ and
$\Sp(n)$. For the full list of exterior pairs in other
coefficients see Corollary~\ref{thm:borel}. Outside exterior pairs, as for
$\mathrm{SO}(n)$ and $G_2$ over $\F_2$, $H_*(G;R)$ is not an
exterior algebra on odd primitives and the mixing terms of
Definition~\ref{def:mixing} are not defined.

To an exterior pair with exponents $m_1,\dots,m_r$ we associate
the product of odd spheres
\[
X_G\;=\;\prod_{i=1}^{r}S^{2m_i+1},
\]
an exponent $0$ giving a circle. Then $X_G$ has the homology of
$G$ over $R$ and the rational homotopy type of $G$. The following
theorem, the main result of this paper, computes the BV
algebra of $G$ in terms of $X_G$ and $\Gamma$, where $\Gamma$ is
the torsion subgroup of $\pi_1(G)$.

\phantomsection\label{intro:thmA}
\begin{thmArep}[BV rigidity; Theorems~\ref{thm:rigid} and~\ref{thm:rigidtorsion}]
Let $(G,R)$ be an exterior pair. There is an isomorphism of
BV algebras
\[
\HH_*(LG;R)\;\cong\;\HH_*(LX_G;R)\otimes R[\Gamma],
\]
with $R[\Gamma]$ the group ring of $\Gamma$, placed in degree
zero, and $\Delta_G=\Delta_{X_G}\otimes1$.
In particular $\HH_*(LG;\Q)\cong\HH_*(LX_G;\Q)\otimes\Q[\Gamma]$
for every connected compact Lie group $G$.
\end{thmArep}

\noindent The right-hand side is computed by Menichi
\cite{Menichi} and stated in Proposition~\ref{prop:spheres}. The proof trivializes the mixing on $K$, the
compact factor of the universal cover of $G$
(Theorem~\ref{thm:rigid}), and the covering
$K\to K/\Gamma$ onto the semisimple part of $G$ then contributes
the factor $R[\Gamma]$ (Theorem~\ref{thm:rigidtorsion}). The isomorphism is in
general not induced by a map of spaces. Over $\Z$, no map between
$G$ and $X_G$ induces it for $\Sp(n)$, $n\geq2$, or $\SU(n)$,
$n\geq3$ (Remark~\ref{rem:sharp}).

Conversely, over an integral domain the graded algebra
$\HH_*(LG;R)$ determines the exponents and $R[\Gamma]$:

\phantomsection\label{intro:thmB}
\begin{thmBrep}[Classification; Corollaries~\ref{cor:classify},
\ref{cor:torsionclass} and~\ref{cor:rationalpi1}]
Let $R$ be an integral domain and $(G,R)$, $(G',R)$ exterior
pairs. The following are equivalent:
\begin{enumerate}
\item $\HH_*(LG;R)\cong\HH_*(LG';R)$ as BV
algebras;
\item $\HH_*(LG;R)\cong\HH_*(LG';R)$ as graded $R$-algebras;
\item $G$ and $G'$ have the same exponents, and
$R[\Gamma]\cong R[\Gamma']$.
\end{enumerate}
When $\pi_1(G)$ and $\pi_1(G')$ are torsion-free, as they are
for $R=\Z$, condition (3) says $G\simeq_\Q G'$; when
$R\subseteq\Q$, it says $G\simeq_\Q G'$ and
$\pi_1(G)\cong\pi_1(G')$.
\end{thmBrep}

\noindent So the rational BV algebra of $LG$ determines
$\pi_1(G)$, but the rational homotopy type of $G$ does not, since
a finite covering is a rational equivalence. Here $\Gamma$ enters
through $H_0(LG;\Q)=\Q[\pi_1(G)]$, and $\Q[\Gamma]$ determines
$\Gamma$ \cite{PerlisWalker}. This is not true when we work over
$\F_p$, where the group ring need not determine the group: for
example, $\F_p[\Z/4]\cong\F_p[\Z/2\times\Z/2]$ if
$p\equiv1\bmod4$.

The isomorphism of \hyperref[intro:thmA]{Theorem~A} can also be
viewed algebraically. Over $\Q$, as we saw above, formality and
the theorem of F\'elix and Thomas \cite{FelixThomas} identify
$\HH_*(LG;\Q)$ with the Hochschild cohomology of $H^*(G;\Q)$ as
BV algebras for simply connected $G$, and the isomorphism is
forced. Formality alone does not suffice over $\Z$ or $\F_p$,
as Menichi's $S^2$ demonstrates over $\F_2$
\cite[Corollary~30]{Menichi}. Still, for exterior pairs the
identification holds over $R$:

\phantomsection\label{intro:thmD}
\begin{thmDrep}[Hochschild cohomology model; Theorem~\ref{thm:hochmodel}]
Let $(G,R)$ be an exterior pair with $G$ simply connected. Then
\[
\HH_*(LG;R)\;\cong\;HH^*\big(H^*(G;R);H^*(G;R)\big)
\]
as BV algebras, the right-hand side carrying
the cup product and the BV operator induced by Connes' operator
through Poincar\'e duality $H^*(G;R)\cong H^*(G;R)^\vee$.
\end{thmDrep}

\noindent For the classifying space $BG$, Kuribayashi and Menichi
\cite{KuribayashiMenichi} prove the analogous statement: for
almost any connected compact Lie group $G$ and $p$ odd or $p=0$,
the BV algebra $\HH^*(LBG;\F_p)$ is the Hochschild cohomology
$HH^*\big(H_*(G;\F_p);H_*(G;\F_p)\big)$, and over $\F_2$ this
fails for $\mathrm{SO}(3)$ and $G_2$; neither
$(\mathrm{SO}(3),\F_2)$ nor $(G_2,\F_2)$ is an exterior pair.

Since the isomorphism of \hyperref[intro:thmD]{Theorem~D} is
abstract, passing through $X_G$, we think a natural map would
realize it. The BV operator on the right-hand side of \hyperref[intro:thmD]{Theorem~D} comes from
Poincar\'e duality, the isomorphism of $H^*(G;R)$-bimodules
\[
\Theta\colon H^*(G;R)\to H^*(G;R)^\vee,\qquad
\Theta(a)(b)=\langle ab,[G]\rangle .
\]
On the cochains $C=C^*(G;R)$ the duality is a derived map.
Hochschild cohomology is an Ext group,
$HH^*(C;C^\vee)=\mathrm{Ext}^*_{C\otimes C^{\mathrm{op}}}(C,C^\vee)$,
so its classes are morphisms $C\to C^\vee$ in the derived
category of $C$-bimodules. Jones' isomorphism
$J\colon H_*(LG;R)\cong HH^*(C;C^\vee)$ \cite{Jones} gives one,
\[
\theta\;:=\;J(c_*[G]),
\]
for $c\colon G\to LG$ the inclusion of the constant loops; it has
degree $d=\dim G$ and induces $\Theta$ on cohomology
(Section~\ref{sub:hochmodel}).

\phantomsection\label{intro:conj}
\begin{conjrep}[Formality with duality; Conjecture~\ref{conj:pdformal}]
Let $(G,R)$ be an exterior pair with $G$ simply connected. There is an
$A_\infty$ quasi-isomorphism $H^*(G;R)\to C^*(G;R)$ carrying $\Theta$
to $\theta$.
\end{conjrep}

\noindent Such a map, composed with the Cohen-Jones isomorphism
\cite[Theorem~3]{CohenJones}, gives
\[
\HH_*(LG;R)\;\cong\;HH^*\big(C^*(G;R);C^*(G;R)\big)\;\cong\;
HH^*\big(H^*(G;R);H^*(G;R)\big),
\]
which realizes \hyperref[intro:thmD]{Theorem~D} once the Cohen-Jones isomorphism is one
of BV algebras: this holds over a field of characteristic zero
\cite{FelixThomas} and is conjectured in general
\cite{TradlerBV,MenichiBV}. We make the conjecture on three
grounds. \hyperref[intro:thmD]{Theorem~D} shows that the target
is the right BV algebra. Over $\Q$ the conjecture holds, since
$H^*(G;\Q)$ with its pairing is a Poincar\'e duality model of
$C^*(G;\Q)$ in the sense of F\'elix and Thomas
(Section~\ref{sub:hochmodel}). And the $A_\infty$ formality of
$C^*(G;R)$ alone, without the condition on $\Theta$, is
established for every exterior pair over a principal ideal domain
of characteristic zero in work in progress of the author; what
remains is to choose the formality map compatibly with the
duality.

We use $\HH_*(L\Sp(n);\Z)$ as the running example for the
rigidity phenomenon of \hyperref[intro:thmA]{Theorem~A}. The
ring and the form of $\Delta$ in Hepworth's presentation were
obtained by Kupers \cite[Theorem~4.1]{Kupers}, with the homology
suspension determined up to sign through the quaternionic Stiefel
manifolds, and the factor $2$ of the coproduct of Kono and
Kozima, which he quotes, is missing from his formula for
$\Delta$. We compute the suspension through the edge map of the
cobar spectral sequence, which fixes the sign, and
\hyperref[intro:thmC]{Theorem~C} restores the coefficient and
adds the Gerstenhaber bracket (Remark~\ref{rem:kupers}).

\phantomsection\label{intro:thmC}
\begin{thmCrep}[The BV algebra of $L\Sp(n)$; Theorem~\ref{thmC}]
There is an isomorphism of graded algebras
\begin{gather*}
\HH_*(L\Sp(n);\Z)\;\cong\;
\Z[z_1,z_3,\dots,z_{2n-1}]\otimes
\Lambda_{\Z}(\alpha_3,\alpha_7,\dots,\alpha_{4n-1}),\\
|z_{2k-1}|=4k-2,\qquad |\alpha_{4k-1}|=-(4k-1),
\end{gather*}
under which the BV operator is
\[
\Delta\;=\;\sum_{k=1}^{n}\ \partial_k\otimes
\frac{\partial}{\partial\alpha_{4k-1}},
\qquad
\partial_k\;=\;\frac{\partial}{\partial z_{2k-1}}
\;+\;2\sum_{l=k+1}^{n}z_{2(l-k)}\,
\frac{\partial}{\partial z_{2l-1}} .
\]
The induced Gerstenhaber bracket is a derivation in each variable,
determined by
\[
\{z_{2l-1},\alpha_{4k-1}\}=\delta_{kl}+2z_{2(l-k)},\qquad
\{z,z'\}=\{\alpha,\alpha'\}=0 .
\]
\end{thmCrep}

\noindent Its mixing term has every coefficient equal to $2$.
We first exhibit the rigidity phenomenon for $\Sp(n)$ in
Example~\ref{ex:spshear}.

Section~\ref{sec:hepworth} recalls Hepworth's presentation of
$\HH_*(LG;R)$. Section~\ref{sec:groups} determines
the exterior pairs and the groups in play. Sections~\ref{sub:intersection}
to~\ref{sec:suspension} compute the three inputs of the
presentation: the intersection ring, the Pontryagin Hopf algebra,
and the homology suspension. Section~\ref{sec:hopf} proves
\hyperref[intro:thmC]{Theorem~C}. Section~\ref{sec:rigidity} removes the mixing and
proves \hyperref[intro:thmA]{Theorems~A} and~\hyperref[intro:thmB]{B} for $\pi_1(G)$ torsion-free, and
\hyperref[intro:thmD]{Theorem~D}; Section~\ref{sub:torsion} extends \hyperref[intro:thmA]{Theorems~A} and~\hyperref[intro:thmB]{B} to
torsion in $\pi_1(G)$.

\medskip\noindent\textit{Acknowledgements.} The author thanks his
supervisor Somnath Basu for pointing out the rigidity phenomenon
for $\SU(n)$, which led to this work, and the Council of
Scientific and Industrial Research, India, for the CSIR-UGC
research fellowship that supported it. An AI assistant was used
for language editing; all mathematical content and all writing
decisions are the author's own.

\section{The BV algebra of Lie groups}\label{sec:hepworth}\label{sub:hepworth}

Let $G$ be a compact Lie group of dimension $d$, and take
coefficients in a commutative ring $R$. We write
$\HH_*(G)=H_{*+d}(G)$ for the intersection ring, which is the Poincar\'e
dual of the cup product, and
$\HH_*(LG)=H_{*+d}(LG)$ for the loop homology. The free loop space of a Lie group splits through right translation,
\[
LG\;\xrightarrow{\ \cong\ }\;\loops G\times G,
\qquad \ell\mapsto\bigl(\ell\,\ell(0)^{-1},\,\ell(0)\bigr),
\]
and, since the integral homology of $\loops G$ is torsion-free by
Bott \cite[Theorem~A]{Bott56}, the K\"unneth map
$H_*(\loops G)\otimes H_*(G)\to H_*(\loops G\times G)\cong H_*(LG)$
is an isomorphism of graded $R$-modules. Hepworth
\cite{Hepworth} describes the BV algebra $\HH_*(LG)$ through
this splitting (Theorem~\ref{thm:hepworth1}).

Loop product under this splitting is the tensor product of
the Pontryagin ring of $\loops G$ and the intersection ring of
$G$. The following three structures describes the BV operator:
\begin{enumerate}
\item the coproduct $\psi(a)=\sum a_{(1)}\otimes a_{(2)}$ of
$H_*(\loops G)$, induced by the diagonal;
\item the Pontryagin action
$H_*(G)\otimes\HH_*(G)\to\HH_*(G)$, $(a,x)\mapsto ax$, induced
by the multiplication;
\item the homology suspension
$\sigma\colon H_*(\loops G)\to H_{*+1}(G)$,
$\sigma(a)=\mathrm{ev}_*([S^1]\times a)$, for the evaluation map
$\mathrm{ev}\colon S^1\times\loops G\to G$.
\end{enumerate}

\begin{theorem}[{\cite[Theorem~1]{Hepworth}}]
\label{thm:hepworth1}
For any compact Lie group $G$ there is an isomorphism of graded
algebras $\HH_*(LG)\cong H_*(\loops G)\otimes\HH_*(G)$ with respect
to which the BV operator is
\[
\Delta(a\otimes x)\;=\;\sum a_{(1)}\otimes\sigma(a_{(2)})\,x .
\]
\end{theorem}

\noindent The suspension takes values in the primitives
\cite[Lemma~8]{Hepworth}, and of odd degree since $H_*(\loops G)$ is
even. So the odd primitives are the lens through which we see homology
 suspension and the Pontryagin action on $H_*(G)$.

\begin{definition}[{\cite[Definition~2]{Hepworth}}]
\label{def:hepworth2}
Suppose either that $R$ is a field, or that $R=\Z$ and
$H_*(G;\Z)$ is torsion-free, so that $H_*(G)$ is a Hopf
algebra. Choose a basis $p_1,\dots,p_n$ of the odd-degree part
of the primitive subspace of $H_*(G)$ and classes
$p^i\in H^*(G)$ with $\langle p^i,p_j\rangle=\delta_{ij}$.
Define
\begin{align*}
\partial_i&\colon H_*(\loops G)\to H_{*-|p_i|+1}(\loops G), &
\partial_i(a)&=\sum\big\langle p^i,\sigma(a_{(1)})\big\rangle\,a_{(2)},\\
\delta_i&\colon \HH_*(G)\to\HH_{*+|p_i|}(G), &
\delta_i(x)&=p_i\,x .
\end{align*}
\end{definition}

\noindent Here $\sigma$ acts on the first factor, and in Theorem~\ref{thm:hepworth1}
on the second, both are stated verbatim as in Hepworth. The
difference is immaterial as $\psi$ is cocommutative and $H_*(\loops G)$
is even \cite{Bott56}. Expanding $\sigma(a_{(2)})$ in the basis $p_i$
gives:

\begin{proposition}[{\cite[Proposition~3]{Hepworth}}]
\label{prop:hepworth3}
In the situation of Definition~\ref{def:hepworth2} the operators
$\partial_i$ and $\delta_i$ are derivations, and under the
isomorphism of Theorem~\ref{thm:hepworth1},
$\ \Delta=\sum_{i=1}^{n}\partial_i\otimes\delta_i$.
\end{proposition}

\section{The groups in play}\label{sec:groups}

This paper treats compact Lie groups $G$ whose homology
over $R$ is an exterior algebra on odd primitive generators. We shall call such a pair $(G,R)$ as an exterior pair (Definition~\ref{def:exterior}).
In this section we determine the exterior pairs in Theorem~\ref{thm:whichpairs}, and give a full exposition of which groups qualify in Corollary~\ref{thm:borel} using Borel's theorems \cite{Borel61}
(Theorems~\ref{thm:borel32} and~\ref{thm:borel25}) through the compact factor $K$ of the universal cover of $G$. Every BV algebra computation in later
sections works through $K$ to reach $G$, so we also show, by a
simple lifting argument, that $(K,R)$ is exterior whenever $(G,R)$
is (Corollary~\ref{cor:Kexterior}).

\subsection{Exterior pairs and exponents}\label{sub:exterior}

The hypothesis of Definition~\ref{def:hepworth2} is tightened further: 
the odd primitives should also generate $H_*(G;R)$ as an exterior algebra. 
This will remain our hypothesis throughout the remainder of this paper.

\begin{definition}[Exterior pairs]\label{def:exterior}
The pair $(G,R)$ is called \emph{exterior} if $H_*(G;R)$ is an
exterior Hopf algebra over $R$ on finitely many primitive
generators of odd degree,
\[
H_*(G;R)\;\cong\;\Lambda_R(a_1,\dots,a_r),
\qquad
|a_k|=2m_k+1,\quad m_1\leq\dots\leq m_r.
\]
The multiset $\{m_1,\dots,m_r\}$ is the multiset of
\emph{$R$-exponents} of $G$.
\end{definition}

Rationally the picture is classical. The rational homotopy of $G$ is spherical \cite[Chapitre~V, \S3, Corollaire~2]{Serre53}. So choose a representative $S^{2m_i+1}\to G$ for each of the $r=\mathrm{rank}\,G$ generators. The group multiplication then assembles them into the map
\[
\prod_{i=1}^{r}S^{2m_i+1}\longrightarrow G,
\]
which is an isomorphism on $\pi_*(-)\otimes\Q$ and hence on rational
homology. So the rational homology is primitively generated,
\[
H_*(G;\Q)\;\cong\;
\Lambda_\Q(a_1,\dots,a_r),\qquad |a_k|=2m_k+1,
\]
 the generators being the images of the fundamental classes of the spheres. Hence $(G,\Q)$ is always an exterior pair, with $\{m_1,\dots,m_r\}$ the multiset of \emph{classical exponents} of $G$.

Over $\Z$ and over $\F_p$, being an exterior pair is a genuine condition. For a Hopf algebra $A$ which is finitely generated over $\Z$, the quotient $A/\mathrm{Tor}$
is a primitively generated exterior algebra on generators of odd degree
\cite[Chapter~VII, Theorem~1.14]{MimuraToda}. Over a perfect
field, a Hopf algebra with a simple system of generators of
odd degrees is primitively generated and exterior
\cite[Chapter~VII, Theorem~1.12]{MimuraToda}, which guarantees that an
exterior Hopf algebra on odd generators is primitively generated \cite[Chapter~VII, Corollary~1.13]{MimuraToda}. For a $0$-connected $H$-space $X$ with finitely generated homology, $H^*(X;\Z)/\mathrm{Tor}$ is an exterior algebra on primitive
generators of odd degree, of the same degrees as over $\Q$
\cite[Chapter~VII, Theorem~1.22]{MimuraToda}. Similarly, when
$H_*(X;\Z)$ has no $p$-torsion, $H^*(X;\F_p)\cong(H^*(X;\Z)/\mathrm{Tor})\otimes\F_p$ has the same form as over $\Q$ \cite[Chapter~VII, Corollary~1.23]{MimuraToda}. For homology, the dual of such a
Hopf algebra is exterior on the duals of the generators, which are primitive \cite[Chapter~IV, Theorem~2.17]{MimuraToda}. This discussion carries over to an arbitrary commutative 
ring:

\begin{theorem}[Determination of exterior pairs]\label{thm:whichpairs}
Let $G$ be a compact connected Lie group and $R$ a commutative
ring. Then $(G,R)$ is exterior if and only if $H_*(G;\Z)$ has
no $p$-torsion for any prime $p$ that is not a unit in $R$,
and then the $R$-exponents are the classical exponents. In
particular:
\begin{enumerate}
\item $(G,\Q)$ is exterior;
\item $(G,\Z)$ is exterior if and only if $H_*(G;\Z)$ is
torsion-free;
\item $(G,\F_p)$ is exterior if and only if $H_*(G;\Z)$ has
no $p$-torsion.
\end{enumerate}
\end{theorem}

\begin{proof}
Write $T=\mathrm{Tor}\,H_*(G;\Z)$, a finite abelian group, and $F=H_*(G;\Z)/T$, free of finite rank.

($\Leftarrow$) Let $T\cong\bigoplus_j\Z/q_j^{e_j}$ with every $q_j$ a unit in $R$. Then
\[
\Z/q_j^{e_j}\otimes R=R/q_j^{e_j}R=0,\qquad \mathrm{Tor}^{\Z}_1(\Z/q_j^{e_j},R)=0,
\]
so $T\otimes R=0=\mathrm{Tor}^{\Z}_1(T,R)$, and the universal coefficient sequence
\[
0\to H_*(G;\Z)\otimes R\to H_*(G;R)\to\mathrm{Tor}^{\Z}_1(H_{*-1}(G;\Z),R)\to0
\]
gives $H_*(G;R)\cong H_*(G;\Z)\otimes R\cong F\otimes R$. By \cite[Chapter~VII, Theorem~1.22; Chapter~IV, Theorem~2.17]{MimuraToda},
\[
F\cong\Lambda_\Z(a_1,\dots,a_r),\qquad |a_k|=2m_k+1,\quad a_k\ \text{primitive},
\]
with $\{m_k\}$ the classical exponents; hence $H_*(G;R)\cong\Lambda_R(a_1\otimes1,\dots,a_r\otimes1)$ on primitive generators, and $(G,R)$ is exterior with the classical exponents.

($\Rightarrow$) Let $(G,R)$ be exterior and $p$ a prime that is not a unit in $R$. Choose a maximal ideal $\mathfrak m\ni p$ and put $k=R/\mathfrak m$, a field of characteristic $p$. Since $H_*(G;R)\cong\Lambda_R(a_1,\dots,a_r)$ is free over $R$, the universal coefficient sequence over $R$ gives
\[
H_*(G;k)\cong H_*(G;R)\otimes_Rk\cong\Lambda_k(a_1,\dots,a_r),\qquad
H^*(G;k)\cong\Lambda_k(a_1,\dots,a_r)^\vee\cong\Lambda_k(a_1^\vee,\dots,a_r^\vee),
\]
exterior algebras on generators of odd degree. By \cite[Chapter~VII, Theorem~2.14]{MimuraToda}, $H_*(G;\Z)$ has no $p$-torsion.
\end{proof}

\noindent The $R$-exponents of an exterior pair are thus the
classical exponents of $G$, and we call them the
\emph{exponents} of $G$.

\subsection{The census}\label{sub:census}

Theorem~\ref{thm:whichpairs} has turned the determination of the exterior pairs into a question about $p$-torsion in $H_*(G;\Z)$. We further reduce this question to the simple simply connected groups, where Borel's torsion theorem settles it. First, a connected
compact Lie group splits into a torus and a semisimple group:

\begin{theorem}[{\cite[Corollaire~3.2]{Borel61}}]
\label{thm:borel32}
Every connected compact Lie group is homeomorphic to the product
of its semisimple part $G_{ss}$ and the identity component of its
centre.
\end{theorem}

\noindent Here the identity component of the centre is the torus
$T^k$. The splitting being smooth, $G$ is diffeomorphic to
$T^k\times G_{ss}$. The semisimple group $G_{ss}$ has compact and
semisimple universal cover $K$
\cite[Chapter~V, Theorem~5.29]{MimuraToda}. Let
$\Gamma=\pi_1(G_{ss})$, then $G_{ss}=K/\Gamma$, with $\Gamma$ a
finite, central, discrete normal subgroup of the connected group
$K$. Hence $\mathbb{R}^k\times K$ is the universal cover of $G$,
and $\pi_1(G)=\Z^k\times\Gamma$, so $\Gamma$ is the torsion
subgroup of $\pi_1(G)$. By Cartan's
classification
$K=K_1\times\cdots\times K_s$, with each $K_i$ simple, simply
connected and isomorphic to one of
\[
\SU(n)\ (n\geq2),\quad \Sp(n)\ (n\geq1),\quad
\mathrm{Spin}(n)\ (n\geq7),\quad
G_2,\ F_4,\ E_6,\ E_7,\ E_8
\]
(see for example \cite[Chapter~V, \S6]{MimuraToda}, the low spin
groups reappear in the first two families,
$\mathrm{Spin}(3)\cong\SU(2)$, $\mathrm{Spin}(5)\cong\Sp(2)$,
$\mathrm{Spin}(6)\cong\SU(4)$).

If $(G,R)$ is exterior then $|\Gamma|$ is a unit in $R$
(Lemma~\ref{lem:gammaunit} below), and then
$H_*(K;R)\cong H_*(G_{ss};R)$ (Proposition~\ref{prop:covering}
below). So it is enough to know the homology of the simply
connected cover $K=K_1\times\cdots\times K_s$. If $H_*(K_i;\Z)$ has $p$-torsion
then $H_*(K;\Z)$ has $p$-torsion, as each $K_i$ is a
retract of $K$. Suppose no $H_*(K_i;\Z)$ has
$p$-torsion. Then the $H_*(K_i;\Z_{(p)})$ are free over $\Z_{(p)}$, the integers localized at $p$, and the K\"unneth theorem gives
$H_*(K;\Z_{(p)})=\bigotimes_iH_*(K_i;\Z_{(p)})$, the Tor terms vanishing. By Theorem~\ref{thm:whichpairs} each $(K_i,\Z_{(p)})$ is an exterior pair, and hence so is $(K,\Z_{(p)})$. Since $p$ is the only prime that is not a unit in $\Z_{(p)}$, Theorem~\ref{thm:whichpairs} gives that $H_*(K;\Z)$ has no $p$-torsion. So $(K,R)$ is exterior if
and only if every $(K_i,R)$ is. The question thus reduces to
the simply connected simple groups $K_i$, and Borel's torsion
theorem settles it:

\begin{theorem}[{\cite[Th\'eor\`eme~2.5]{Borel61}}]
\label{thm:borel25}
Let $G$ be a simple, simply connected compact Lie group and $p$ a
prime. Then $H_*(G;\Z)$ has $p$-torsion exactly in the following
cases: $p=2$ and $G=\mathrm{Spin}(n)$ $(n\geq7)$, $G_2$, $F_4$,
$E_6$, $E_7$, $E_8$; $p=3$ and $G=F_4$, $E_6$, $E_7$, $E_8$;
$p=5$ and $G=E_8$.
\end{theorem}

\begin{proposition}[Exterior pairs among the simple groups]
\label{prop:census}
Let $G$ be a simple simply connected compact Lie group and $R$
a commutative ring. Then $(G,R)$ is exterior if and only if
every torsion prime of $G$ is a unit in $R$, the torsion
primes being
\begin{center}
\begin{tabular}{ll}
group & torsion primes\\[2pt]
\hline
$\SU(n)$, $\Sp(n)$ & none\\
$\mathrm{Spin}(n)$ $(n\geq7)$, $G_2$ & $2$\\
$F_4$, $E_6$, $E_7$ & $2,3$\\
$E_8$ & $2,3,5$
\end{tabular}
\end{center}
In particular $(G,\Q)$ is always exterior; $(G,\Z)$ is
exterior only for $\SU(n)$ and $\Sp(n)$; and $(G,\F_p)$ is
exterior if and only if $p$ is not a torsion prime of $G$.
\end{proposition}

\begin{proof}
By Theorem~\ref{thm:whichpairs}, $(G,R)$ is exterior if and
only if $H_*(G;\Z)$ has no $p$-torsion for any prime $p$ that
is not a unit in $R$, and Theorem~\ref{thm:borel25} lists the
primes at which $H_*(G;\Z)$ has torsion.
\end{proof}

The failure of $(G,R)$ to be an exterior pair comes from
torsion in $H_*(G;\Z)$, in two degrees: in degree one
(Remark~\ref{rem:pi1tors}) and in higher degrees
(Remark~\ref{rem:torsionfail}).

\begin{remark}[Torsion in $\pi_1$]\label{rem:pi1tors}
Torsion in $\pi_1(G)$ gives torsion in $H_1(G;\Z)\cong\pi_1(G)$,
so $(G,\Z)$ is not an exterior pair. 
Over $\F_p$ with $p\nmid|\Gamma|$,
Proposition~\ref{prop:covering} below gives
$H_*(G_{ss};\F_p)\cong H_*(K;\F_p)$, and $(G,\F_p)$ is exterior
exactly when each simple factor of $K$ is. When
$p\mid|\Gamma|$ the same is not true. For example, for $\mathrm{SO}(3)=\SU(2)/(\Z/2)$
at $p=2$, $(\SU(2),\F_2)$ is exterior, but $H^*(\mathrm{SO}(3);\F_2)\cong\F_2[x]/(x^4)$ is not exterior, $x^2$ being an even primitive.
\end{remark}

\begin{remark}[Torsion in higher degrees]\label{rem:torsionfail}
The simply connected groups $\mathrm{Spin}(n)$ $(n\geq7)$,
$G_2$, $F_4$, $E_6$, $E_7$, $E_8$ each have torsion in $H_i$ for
some $i\geq2$ (Theorem~\ref{thm:borel25}), even though $\pi_1$ is trivial. Over $\F_p$, $p$ a torsion prime, they are not exterior. For $\mathrm{Spin}(n)$ at $p=2$
the degree-three generator of $H^*(\mathrm{Spin}(n);\F_2)$ has
nonzero square \cite[Chapter~IV, Theorem~2.19]{MimuraToda}, an
even primitive, and similarly for each exceptional group
\cite[Chapter~VII, Section~6]{MimuraToda}, and at $p=3$ and $p=5$ the exceptional groups with torsion have
truncated polynomial generators of even degree
\cite[Chapter~VII, Theorems~5.14--5.16]{MimuraToda}.
\end{remark}

\subsection{The universal cover}\label{sub:cover}
By Theorem~\ref{thm:borel32} and the discussion following it, $G\cong T^k\times G_{ss}\cong T^k\times K/\Gamma$. We will first note that $|\Gamma|$ is a unit in $R$ whenever $(G,R)$ is an exterior pair (Lemma~\ref{lem:gammaunit}). The covering $K\to K/\Gamma$ then induces an isomorphism of Hopf algebras on $R$-homology (Proposition~\ref{prop:covering}), so $(K,R)$ is an exterior pair whenever $(G,R)$ is (Corollary~\ref{cor:Kexterior}). Using the Borel characterization (Proposition~\ref{prop:census}) this allows us to conclude which groups are exterior (Corollary~\ref{thm:borel}).

\begin{lemma}[The order of $\Gamma$]\label{lem:gammaunit}
If $(G,R)$ is an exterior pair, then $|\Gamma|$ is a unit in
$R$.
\end{lemma}

\begin{proof}
Let $p\nmid|\Gamma|$. Since $G\cong T^k\times G_{ss}$ with $H_*(T^k;\Z)$ free,
$H_1(G_{ss};\Z)=\Gamma$ is a direct summand of $H_1(G;\Z)$. So by theorem 3.2 $p$ is a unit in $R$. A prime dividing $|\Gamma|$ is therefore a
unit in $R$, and so is $|\Gamma|$.
\end{proof}

\begin{lemma}[Translations]\label{lem:translation}
For a path-connected topological group $K$ and $g\in K$, the
translation $x\mapsto gx$ is homotopic to the identity, along
a path from $1$ to $g$. So the deck transformations of
$p\colon K\to G_{ss}$ act trivially on $H_*(K;R)$ and preserve
orientation.
\end{lemma}

\begin{proposition}[The covering]\label{prop:covering}
If $|\Gamma|$ is invertible in $R$, then
$p_*\colon H_*(K;R)\to H_*(G_{ss};R)$ is an isomorphism of
Hopf algebras.
\end{proposition}

\begin{proof}
Since $\Delta^n$ is simply connected, a singular simplex
$\sigma\colon\Delta^n\to G_{ss}$ has $|\Gamma|$ lifts to $K$,
one through each point of the fibre over $\sigma(e_0)$. Define
the transfer
\[
s_\#\colon C_n(G_{ss};R)\to C_n(K;R),\qquad
s_\#(\sigma)=\sum_{\tilde\sigma}\tilde\sigma,
\]
the sum over the lifts of $\sigma$. As $\tilde\sigma$ runs over
the lifts of $\sigma$, the face $\tilde\sigma\circ d_j$ runs
over the lifts of $\sigma\circ d_j$, a lift being determined by
one value; hence
\[
\partial s_\#(\sigma)
=\sum_j(-1)^j\sum_{\tilde\sigma}\tilde\sigma\circ d_j
=\sum_j(-1)^j\,s_\#(\sigma\circ d_j)
=s_\#(\partial\sigma),
\]
and $s_\#$ is a chain map. Each lift projects to $\sigma$, and
the lifts of $p\circ\tilde\sigma$ are the translates
$g\tilde\sigma$, $g\in\Gamma$, so
\[
p_\#s_\#=|\Gamma|,\qquad s_\#p_\#=\sum_{g\in\Gamma}g_\# .
\]
The translations act trivially on homology
(Lemma~\ref{lem:translation}), so $p_*s_*=|\Gamma|=s_*p_*$, and
$p_*$ is an isomorphism with inverse $|\Gamma|^{-1}s_*$, of Hopf
algebras since $p$ is a homomorphism.
\end{proof}

\begin{corollary}[The universal cover is exterior]\label{cor:Kexterior}
If $(G,R)$ is an exterior pair, so is $(K,R)$.
\end{corollary}

\begin{proof}
Since $G\cong T^k\times G_{ss}$ with $H_*(T^k;\Z)$ free,
$H_*(G;\Z)\cong H_*(T^k;\Z)\otimes H_*(G_{ss};\Z)$ as graded
$\Z$-modules, and the two have the same torsion; so $(G_{ss},R)$ is exterior together
with $(G,R)$, the criterion of Theorem~\ref{thm:whichpairs}
being a condition on the torsion. By
Lemma~\ref{lem:gammaunit}, $|\Gamma|$ is a unit in $R$, and
Proposition~\ref{prop:covering} gives
$H_*(K;R)\cong H_*(G_{ss};R)$ as Hopf algebras.
\end{proof}

Theorem~\ref{thm:whichpairs}, Proposition~\ref{prop:census},
Proposition~\ref{prop:covering} and Corollary~\ref{cor:Kexterior}
give the list of exterior pairs.

\begin{corollary}[The groups in play]\label{thm:borel}
Let $G\cong T^k\times K/\Gamma$ be a connected compact Lie group,
$K=K_1\times\cdots\times K_s$ with $K_i$ simple, and $R$ a
commutative ring. Then $(G,R)$ is an exterior pair if and only
if $|\Gamma|$ is a unit in $R$ and every $(K_i,R)$ is an
exterior pair. In particular:
\begin{enumerate}
\item $(G,\Q)$ is an exterior pair.
\item $(G,\Z)$ is an exterior pair if and only if $H_*(G;\Z)$ is
torsion-free, if and only if $\Gamma=1$ and every $K_i$ is an
$\SU(n)$ or an $\Sp(n)$; $G$ is then diffeomorphic to
$T^k\times\prod_i\SU(n_i)\times\prod_j\Sp(m_j)$.
\item $(G,\F_p)$ is an exterior pair if and only if
$p\nmid|\Gamma|$ and every $(K_i,\F_p)$ is
(Proposition~\ref{prop:census}).
\end{enumerate}
\end{corollary}

\section{The intersection side}\label{sub:intersection}

In this section, we will determine the intersection ring
$\HH_*(G;R)$ of an exterior pair
(Definition \ref{def:exterior}), as well as the Pontryagin action $\delta_i$ on it (Lemma \ref{lem:intersection}), which are needed for Definition \ref{def:hepworth2}. Then, we will compare the intersection rings of $K$ and $G_{ss}=K/\Gamma$, along with their respective Pontryagin actions, using the covering map (Lemma \ref{lem:intquot}).
\subsection{The intersection ring of an exterior pair}\label{sub:intext}

We describe the intersection ring of an exterior pair $(G,R)$ in
terms of the monomials in the generators of $H_*(G;R)$ and their
duals in $H^*(G;R)$.

\begin{definition}[Monomial notation]
\label{rem:symbols}
Let $(G,R)$ be an exterior pair, $H_*(G;R)=\Lambda_R(a_1,\dots,a_r)$
as in Definition~\ref{def:exterior}, and $I=\{1,\dots,r\}$. For
$S\subseteq I$ let
\[
a_S:=\textstyle\prod_{i\in S}a_i,
\]
the product in increasing order of the indices, and
$a_\varnothing=1$. The set $\{a_S\}_{S\subseteq I}$ is an $R$-basis of
$H_*(G;R)$, and since the $a_i$'s are primitive,
\[
\psi(a_S)=\sum_{S=S'\sqcup S''}\pm\,a_{S'}\otimes a_{S''},
\]
with the Koszul signs of the shuffles. Let $x_S\in H^*(G;R)$ be
the dual basis under the Kronecker pairing,
$\langle x_S,a_T\rangle=\delta_{ST}$. Then
\[
H^*(G;R)=\Lambda_R(x_1,\dots,x_r),
\qquad
x_S=\textstyle\prod_{i\in S}x_i .
\]
\end{definition}

\noindent For every exterior pair $(G,R),$ we know that $H_*(G;R)$ is a Hopf algebra. Therefore, Definition~\ref{def:hepworth2} and Proposition~\ref{prop:hepworth3} can be applied with $p_i=a_i$ and $p^i=x_i.$

\begin{definition}[Intersection classes]
\label{def:alpha}
We keep the hypothesis and notations from the previous Definition~\ref{rem:symbols}. Let $d=\dim G$. We orient $G$ by $[G]:=a_I$, and let
$\HH_*(G;R)=H_{*+d}(G;R)$ be the intersection ring,
$(x\cap[G])\cdot(y\cap[G])=(xy)\cap[G]$. For $S\subseteq I$ set
\[
\alpha_S\;:=\;x_S\cap[G]\;\in\;\HH_{-|a_S|}(G;R),
\qquad
\alpha_i\;:=\;\alpha_{\{i\}} .
\]
\end{definition}

\noindent Then $\alpha_S=\alpha_{i_1}\cdots\alpha_{i_k}$ for
$S=\{i_1<\dots<i_k\}$, $\alpha_\varnothing=[G]$ is the unit, and
with the cap product $\langle y,x\cap c\rangle=\langle xy,c\rangle$,
\[
\langle x_T,\alpha_S\rangle
=\langle x_Sx_T,\,a_I\rangle
=\pm\,\delta_{T,\,I\smallsetminus S},
\qquad\text{so}\qquad
\alpha_S=\pm\,a_{I\smallsetminus S}.
\]
Hence $\{\alpha_S\}_{S\subseteq I}$ forms a basis and
\[
\HH_*(G;R)\;=\;\Lambda_R(\alpha_1,\dots,\alpha_r).
\]
For $S=\{i\}$,
$\langle x_{I\smallsetminus\{i\}},\alpha_i\rangle
=\langle x_ix_{I\smallsetminus\{i\}},a_I\rangle=(-1)^{i-1}$, so
$\alpha_i=(-1)^{i-1}a_{I\smallsetminus\{i\}}$, and
$a_ia_{I\smallsetminus\{i\}}=(-1)^{i-1}a_I$ gives
\[
a_i\alpha_i\;=\;[G].
\]
Note that the sign in $\alpha_S=\pm\,a_{I\smallsetminus S}$ 
is determined by the Koszul sign for reordering odd classes and does not play any role. A similar presentation of the case for special unitary groups is found in Tamanoi \cite[Lemma~2-3]{Tamanoi}.

\begin{lemma}[Pontryagin action on the intersection ring]\label{lem:intersection}
Let $(G,R)$ be an exterior pair. For the basis $p_i=a_i$,
Hepworth's derivations $\delta_i$ of
Definition~\ref{def:hepworth2} are the partial derivatives of
$\HH_*(G;R)=\Lambda_R(\alpha_1,\dots,\alpha_r)$, the derivations
of degree $|a_i|$ with
$\partial\alpha_j/\partial\alpha_i=\delta_{ij}$:
\[
\delta_i\;=\;a_i\,(-)\;=\;\frac{\partial}{\partial\alpha_i}.
\]
\end{lemma}

\begin{proof}
The Pontryagin action $\delta_i$ of $a_i$ is a derivation of
degree $|a_i|$ (Proposition~\ref{prop:hepworth3}). On the
generators, $a_i\alpha_j=0$ for $j\neq i$, the factor $a_i$
repeating, and $a_i\alpha_i=[G]$ (Definition~\ref{def:alpha}).
These are the values of $\partial/\partial\alpha_i$, and two
derivations of equal degree agreeing on the generators are
equal.
\end{proof}

\subsection{The quotient}\label{sub:intquot}

The covering $p\colon K\to G_{ss}$ of Section~\ref{sub:cover}
carries the intersection ring along, up to the factor
$|\Gamma|$.

\begin{lemma}[The intersection ring of the quotient]\label{lem:intquot}
Let $|\Gamma|$ be invertible in $R$ and $H_*(K;R)$ free, and
orient $G_{ss}$ so that $p$ preserves orientation. Then
\[
\varphi\;:=\;|\Gamma|^{-1}\,p_*\colon
\HH_*(K;R)\longrightarrow\HH_*(G_{ss};R)
\]
is an isomorphism of intersection rings with
$\varphi[K]=[G_{ss}]$, and
$\varphi(ax)=p_*(a)\,\varphi(x)$ for
$a\in H_*(K;R)$ and $x\in\HH_*(K;R)$.
\end{lemma}

\begin{proof}
The map $p$ is a covering of degree $|\Gamma|$ whose deck
transformations, the translations by $\Gamma$, preserve
orientation (Lemma~\ref{lem:translation}); hence
$p_*[K]=|\Gamma|\,[G_{ss}]$ and
$\varphi[K]=[G_{ss}]$. By Proposition~\ref{prop:covering},
$p_*$ is an isomorphism on homology, so $\varphi$ is
bijective, and $p^*$ is an isomorphism of cohomology rings,
the homologies being free.

Every class of $\HH_*(K;R)$ is $p^*\xi\cap[K]$ for some
$\xi\in H^*(G_{ss};R)$. Let $x=p^*\xi\cap[K]$ and
$y=p^*\eta\cap[K]$, so that $x\cdot y=p^*(\xi\cup\eta)\cap[K]$,
the intersection product being dual to the cup product. The
projection formula gives
\[
p_*\big(p^*\zeta\cap[K]\big)=\zeta\cap p_*[K]
=|\Gamma|\,\zeta\cap[G_{ss}]
\]
for every $\zeta\in H^*(G_{ss};R)$; applied to $\xi\cup\eta$,
$\xi$ and $\eta$ it reads
$p_*(x\cdot y)=|\Gamma|^{-1}\,p_*x\cdot p_*y$, that is,
$\varphi(x\cdot y)=\varphi(x)\varphi(y)$.

The Pontryagin action is induced by the multiplication, and
$p$ is a homomorphism, so $p_*(ax)=p_*(a)\,p_*(x)$; hence
$\varphi(ax)=|\Gamma|^{-1}p_*(a)\,p_*(x)=p_*(a)\,\varphi(x)$.
\end{proof}

\section{The loop side}\label{sub:frame}

In this section we compute the Pontryagin ring
$H_*(\loops G;R)$, which is the loop factor of
Theorem~\ref{thm:hepworth1}. After splitting the torus factor from $G$, the
covering $K\to G_{ss}$ gives an isomorphism of Hopf algebras
$H_*(\loops G_{ss};R)\cong R[\Gamma]\otimes H_*(\loops K;R)$
(Proposition~\ref{prop:loopsplit}), and we use the cobar construction
to show $H_*(\loops K;R)$ is a polynomial ring
(Proposition~\ref{prop:frame}).

\subsection{The loop space of a compact group}\label{sub:loopspace}

By Theorem~\ref{thm:borel32}, $G$ is diffeomorphic to
$T^k\times G_{ss}$. Loops commute with products and
$H_*(\loops T^k;R)$ is free, so the K\"unneth map gives
\[
H_*(\loops G;R)\;\cong\;H_*(\loops T^k;R)\otimes
H_*(\loops G_{ss};R).
\]
The torus factor is immediate: $\loops T^k\simeq\pi_1(T^k)=\Z^k$,
so
\[
H_*(\loops T^k;R)\;=\;R[\Z^k],
\]
in degree zero. The semisimple factor is
$G_{ss}=K/\Gamma$, with $K$ compact simply connected and
$\Gamma=\pi_1(G_{ss})=p^{-1}(1)$ finite, for the covering
$p\colon K\to G_{ss}$ (Section~\ref{sub:cover}).

We use the covering $p$ to describe $\loops G_{ss}$ in terms of
$\loops K$ and $\Gamma$. Let $P_\Gamma K$
be the space of paths $\omega\colon[0,1]\to K$ with $\omega(0)=1$
and $\omega(1)\in\Gamma$, with the compact-open topology and the
pointwise product and inverse
\[
(\omega\cdot\omega')(t)=\omega(t)\,\omega'(t),
\qquad
\omega^{-1}(t)=\omega(t)^{-1},
\]
since $\Gamma$ is a subgroup, so $P_\Gamma K$ is a topological group.
On the loop space of a topological group the pointwise product
and the concatenation are homotopic, so both give rise to the
Pontryagin product.

Composition with $p$ is a continuous homomorphism
$P_\Gamma K\to\loops G_{ss}$, $\omega\mapsto p\circ\omega$,
bijective with continuous inverse by the uniqueness of lifts:
\[
\loops G_{ss}\;\cong\;P_\Gamma K
\]
as topological groups. For $g\in\Gamma$ let $P_gK$ be the paths
ending at $g$, so $P_1K=\loops K$. The $P_gK$ are open and closed, being the fibres of the
endpoint map to $\Gamma$, and path-connected as $K$ is simply
connected. Therefore, they are each the path components of $\loops G_{ss}$.

Choose for each $g\in\Gamma$ a path $c_g$ from $1$ to $g$, with
$c_1$ constant. Right multiplication by $c_g$ and by $c_g^{-1}$
are mutually inverse homeomorphisms $\loops K\rightleftarrows P_gK$,
so
\[
\Phi\colon\loops K\times\Gamma\longrightarrow P_\Gamma K\cong\loops G_{ss},
\qquad(\omega,g)\longmapsto\omega\cdot c_g,
\]
is a homeomorphism. The map $\Phi$ is not a homomorphism, since
$\Phi(\omega,g)\Phi(\omega',h)=\omega c_g\omega'c_h$ while
$\Phi(\omega\omega',gh)=\omega\omega'c_{gh}$, but it is one up to
homotopy: for fixed $g,h\in\Gamma$ the two are homotopic as maps
$\loops K\times\loops K\to P_{gh}K$,
\[
H(s,\omega,\omega')(t)\;=\;\omega(t)\cdot
c_g(st)\,\omega'(t)\,c_g(st)^{-1}\cdot c_g(t)\,c_h(t),
\qquad 0\leq s\leq1,
\]
is a homotopy from $\omega\omega'c_gc_h$ to
$\omega c_g\omega'c_h$, and a path in $P_{gh}K$ from $c_{gh}$ to
$c_gc_h$, multiplied on the left by $\omega\omega'$, is a
homotopy from $\omega\omega'c_{gh}$ to $\omega\omega'c_gc_h$.
Hence $\Phi_*$ is multiplicative, which gives:

\begin{proposition}[The loop space of the semisimple part]\label{prop:loopsplit}
For every commutative ring $R$, $\Phi$ establishes an isomorphism
of Hopf algebras
\[
R[\Gamma]\otimes H_*(\loops K;R)\;\xrightarrow{\ \cong\ }\;
H_*(\loops G_{ss};R),
\]
where $R[\Gamma]$ denotes the group-ring of $\Gamma$ in degree zero.
\end{proposition}

\noindent Putting the two factors together,
\[
H_*(\loops G;R)\;\cong\;R[\Z^k]\otimes R[\Gamma]\otimes
H_*(\loops K;R)
\]
as graded $R$-algebras, the first two factors in degree zero,
\[
R[\Z^k]\otimes R[\Gamma]\;\cong\;R[\Z^k\times\Gamma]\;\cong\;R[\pi_1(G)].
\]
\subsection{The cobar construction}\label{sub:cobar}

Now it remains to calculate $H_*(\loops K;R)$ which Adams' cobar construction will do \cite{Adams1956}.

Let $C$ be a differential graded coalgebra over $R$, free as an
$R$-module and connected, $C_0=R$, with differential $\partial$
and reduced coproduct $\bar\Delta(c)=\sum c'\otimes c''$ on
$\bar C=C_{\geq1}$. The cobar construction $\widehat\Omega C$
is the tensor algebra $T(s^{-1}\bar C)$ on the desuspension of
$\bar C$, that is, each element of $\widehat{\Omega}C$
can be written as a sum of words $\langle c_1|\cdots|c_s\rangle=s^{-1}c_1\otimes\cdots\otimes s^{-1}c_s$
of length $s$ and degree $\sum_i|c_i|-s$, its product is
concatenation, and its differential $d=d_0+d_1$ is the
derivation of the tensor algebra determined on letters by
\[
d_0\langle c\rangle=-\langle\partial c\rangle,
\qquad
d_1\langle c\rangle=\sum(-1)^{|c'|}\langle c'|c''\rangle .
\]
Both have degree $-1$; $d_0$ preserves the length of words and
$d_1$ raises it by one.

Let $X$ be simply connected and $C$ the normalized chains on the
singular simplices $\Delta^n\to X$ such that all vertices and edges of $\Delta^n$
are mapped into the basepoint, which has a coalgebra structure induced by the Alexander-Whitney diagonal. 
These normalized chains compute the homology
$H_*(X;R)$, since $X$ is simply connected \cite[Section~6.8 and the remarks following Theorem~10.16.2]{Neisendorfer}, and
$C_0=R$, $C_1=0$. Adams'
theorem is an isomorphism of graded algebras
\[
H\big(\widehat\Omega C\big)\;\cong\;H_*(\loops X;R).
\]
Every letter has degree at least one, so a word of length $s$
has degree at least $s$: the filtration by word length is
finite in each degree and yields a convergent multiplicative
spectral sequence. When $H_*(X;R)$ is free its $E_2$-page is
$\mathrm{Cotor}^{H_*(X;R)}(R,R)$, which is the homology of the cobar
construction of the homology coalgebra.

Let $(G,R)$ be an exterior pair. Then
$(K,R)$ is an exterior pair (Corollary~\ref{cor:Kexterior}) and
\[
H_*(K;R)\;=\;\Lambda_R(a_1,\dots,a_r),
\]
the $a_k$ primitive in odd degree of at least three since
$H_1(K;R)=0$. For $X=K$ the $E_2$-page of the word-length
spectral sequence above, converging to $H_*(\loops K;R)$, is
therefore the Cotor of the exterior coalgebra
$\Lambda_R(a_1,\dots,a_r)$, which is polynomial: this is the
Koszul duality between exterior and polynomial algebras
\cite[Example~2.2(2)]{Priddy}.

\begin{lemma}[Cotor of an exterior coalgebra]\label{lem:cotorext}
Let $\Lambda=\Lambda_R(a_1,\dots,a_r)$ with the $a_i$ primitive
of odd degrees $n_i\geq3$, viewed as a coalgebra. Then, under
concatenation,
\[
H\big(\widehat\Omega\Lambda\big)\;\cong\;R[y_1,\dots,y_r],
\qquad y_i=\big[\langle a_i\rangle\big],
\]
with $y_i$ of word length one and even total degree $n_i-1$.
\end{lemma}

\begin{proof}
Every letter $\langle a_i\rangle$ is a cycle, the letters being
primitive. For $i\neq j$
\[
d\langle a_ia_j\rangle
\;=\;\pm\big(\langle a_i|a_j\rangle-\langle a_j|a_i\rangle\big).
\]
The classes $[\langle a_i\rangle]$ therefore commute, and
$y_i\mapsto[\langle a_i\rangle]$ extends to an algebra map
$R[y_1,\dots,y_r]\to H(\widehat\Omega\Lambda)$ which sends the
monomial $y_{i_1}\cdots y_{i_s}$ to the class of the word
$\langle a_{i_1}|\cdots|a_{i_s}\rangle$; it is an isomorphism
by the Koszul duality cited above.
\end{proof}

\begin{proposition}[The Pontryagin ring of $\loops K$]\label{prop:frame}
Let $K$ be a simply connected compact Lie group with $(K,R)$
exterior, $H_*(K;R)=\Lambda_R(a_1,\dots,a_r)$ with the $a_k$
primitive of odd degree at least three. Then
\[
H_*(\loops K;R)\;\cong\;R[y_1,\dots,y_r],
\qquad
|y_k|=|a_k|-1,
\]
with $y_k$ the desuspended class $\langle a_k\rangle$.
\end{proposition}

\begin{proof}
By Lemma~\ref{lem:cotorext} the $E_2$-page is
$R[y_1,\dots,y_r]$ as an algebra, concentrated in even total degrees and free, hence $E_2=E_\infty\cong H_*(\loops K;R)$.
\end{proof}

When $H_*(\loops G;R)$ is a polynomial ring, the BV operator of
Proposition~\ref{prop:hepworth3} reduces to:

\begin{proposition}\label{prop:reduce}
In the situation of Definition~\ref{def:hepworth2}, let
$H_*(\loops G;R)$ be a polynomial ring on generators
$b_1,\dots,b_r$. Then, for any basis $p_1,\dots,p_n$ of the
odd primitives of $H_*(G;R)$ and dual classes $p^i$,
\[
\Delta\;=\;\sum_{i=1}^{n}\sum_{l=1}^{r}\partial_i(b_l)\,
\frac{\partial}{\partial b_l}\otimes\delta_i.
\]
The BV operator is thus determined by the finitely many classes
$\partial_i(b_l)\in H_*(\loops G;R)$, $1\leq i\leq n$,
$1\leq l\leq r$.
\end{proposition}

\begin{proof}
By Proposition~\ref{prop:hepworth3},
$\Delta=\sum_k\partial_k\otimes\delta_k$ with each
$\partial_k$ a derivation of the polynomial ring
$H_*(\loops G;R)$, which is
$\sum_l(\text{its value on }b_l)\,\partial/\partial b_l$.
\end{proof}

\section{The homology suspension}\label{sec:suspension}\label{sub:suspension}

In this section we record some basic properties of the homology
suspension and compute it on $\loops K$ and on $\loops G_{ss}$:
for the simply connected group $K$ it carries the polynomial
generators of $H_*(\loops K;R)$ to a basis of the odd
primitives of $H_*(K;R)$ (Lemma~\ref{lem:edge}); and for
$G_{ss}=K/\Gamma$, under
$H_*(\loops G_{ss};R)\cong R[\Gamma]\otimes H_*(\loops K;R)$
of Proposition~\ref{prop:loopsplit}, it kills $g\otimes1$ and
sends $g\otimes b$ to $p_*\sigma_K(b)$, independently of
$g\in\Gamma$ (Lemma~\ref{lem:suspcomp}).

\begin{lemma}[The suspension in degree zero]\label{lem:suspzero}
For a based space $X$, $\sigma\colon H_0(\loops X;R)\to H_1(X;R)$
sends the class of the component of $\gamma\in\pi_1(X)$ to the
Hurewicz image of $\gamma$.
\end{lemma}

\begin{proof}
The evaluation map restricted to $S^1\times\{\gamma\}$ is
$\gamma$, so
$\sigma[\gamma]=\mathrm{ev}_*([S^1]\times[\gamma])=\gamma_*[S^1]$.
\end{proof}

\noindent For the torus the Hurewicz map
$\Z^k=\pi_1(T^k)\to H_1(T^k;\Z)$ is an isomorphism, so
\[
\sigma\colon H_0(\loops T^k;R)=R[\Z^k]\longrightarrow
H_1(T^k;R)=R\otimes\Z^k,
\qquad
[a]\longmapsto 1\otimes a,
\]
for $[a]$ the class of the component of $a\in\Z^k$.

\begin{lemma}[The suspension of the generators]\label{lem:edge}
Let $K$ be as in Proposition~\ref{prop:frame}. Then
$\sigma(y_k)=a_k$.
\end{lemma}

\begin{proof}
In the spectral sequence of Section~\ref{sub:frame} a class of
filtration one is represented by a word of length one, its
symbol is a desuspended primitive of $H_*(K;R)$, and the edge
map is the homology suspension
\cite[Proposition~10.25.4]{Neisendorfer}.
\end{proof}

\begin{lemma}[The suspension on the components]\label{lem:suspcomp}
Let $G_{ss}=K/\Gamma$ with $|\Gamma|$ invertible in $R$, and
let $r_g=\Phi(-,g)\colon\loops K\to\loops G_{ss}$,
$\omega\mapsto p\circ(\omega\cdot c_g)$
(Section~\ref{sub:loopspace}). Then
$\sigma_{G_{ss}}((r_g)_*1)=0$, and for
$b\in H_*(\loops K;R)$ of positive degree
\[
\sigma_{G_{ss}}\big((r_g)_*b\big)\;=\;p_*\,\sigma_K(b),
\]
independent of $g$.
\end{lemma}

\begin{proof}
For $b=1$: $(r_g)_*1$ is the class of the component of $g$,
and its suspension is the Hurewicz image of $g$ in
$H_1(G_{ss};R)=\Gamma\otimes R$ (Lemma~\ref{lem:suspzero}),
which is zero since $|\Gamma|$ is invertible in $R$.

Let $|b|>0$. Since $p$ is a homomorphism with
$p(c_g(0))=p(c_g(1))=1$, the map $\gamma_g:=p\circ c_g$ is a
loop in $G_{ss}$ at $1$, and
\[
\mathrm{ev}\big(t,r_g(\omega)\big)=p(\omega(t))\,\gamma_g(t).
\]
So $\mathrm{ev}\circ(\mathrm{id}\times r_g)$ is the composite
\[
S^1\times\loops K\xrightarrow{\ D\ }
(S^1\times\loops K)\times S^1
\xrightarrow{\ (p\circ\mathrm{ev}_K)\times\gamma_g\ }
G_{ss}\times G_{ss}\xrightarrow{\ \mu\ }G_{ss},
\qquad D(t,\omega)=((t,\omega),t),
\]
with $\mathrm{ev}_K$ the evaluation of $K$ and $\mu$ the
multiplication. The diagonal of $S^1$ sends $[S^1]$ to
$[S^1]\times[\mathrm{pt}]+[\mathrm{pt}]\times[S^1]$, so
\[
D_*\big([S^1]\times b\big)
=\big([S^1]\times b\big)\times[\mathrm{pt}]
+(-1)^{|b|}\big([\mathrm{pt}]\times b\big)\times[S^1].
\]
The second term dies under $\mathrm{ev}_K$: $\mathrm{ev}_K(0,\omega)=\omega(0)=1$
is constant and $|b|>0$. The first is carried to
$p_*\sigma_K(b)\times[1]$, since
$\mathrm{ev}_{K*}([S^1]\times b)=\sigma_K(b)$ and
$\gamma_{g*}[\mathrm{pt}]=[1]$, and then to
$p_*\sigma_K(b)$ by $\mu$, as $[1]$ is the unit. Hence
$\sigma_{G_{ss}}((r_g)_*b)=\mathrm{ev}_*([S^1]\times(r_g)_*b)
=p_*\sigma_K(b)$.
\end{proof}

Lemma~\ref{lem:sigmaprops} uses a second description of
$\sigma$. In the path fibration
$\loops X\to PX\xrightarrow{\ q\ }X$, $q(\gamma)=\gamma(1)$,
the space $PX$ is contractible, so the boundary
$\partial\colon H_{n+1}(PX,\loops X)\to H_n(\loops X)$ is an
isomorphism for $n\geq1$. For $|e|\geq1$,
\[
\sigma(e)\;=\;q_*(u)\qquad\text{for the unique $u$ with }\partial u=e,
\]
$q$ read as a map of pairs $(PX,\loops X)\to(X,*)$. Indeed the
map of pairs
$c\colon([0,1]\times\loops X,\{0,1\}\times\loops X)\to(PX,\loops X)$,
$c(t,\gamma)(s)=\gamma(st)$, satisfies $q\circ c=\mathrm{ev}$
and $\partial\,c_*([0,1]\times e)=e-\varepsilon(e)[\mathrm{pt}]=e$,
so $u=c_*([0,1]\times e)$ and
$q_*(u)=\mathrm{ev}_*([S^1]\times e)=\sigma(e)$.

\begin{lemma}[The homology suspension kills decomposables]
\label{lem:sigmaprops}
Let $X$ be a based space and $e,e'\in H_*(\loops X;R)$ of
positive degree. Then $\sigma(ee')=0$.
\end{lemma}

\begin{proof}
Let $a\colon\loops X\times PX\to PX$, $a(\omega,\gamma)=\omega*\gamma$,
concatenate a loop in front of a path. It carries
$\loops X\times\loops X$ into $\loops X$, where it is the
concatenation inducing the Pontryagin product, and it
preserves endpoints: $q\circ a=q\circ\mathrm{pr}_2$ for the
projection $\mathrm{pr}_2\colon\loops X\times PX\to PX$. So
$a$ is a map of pairs
$(\loops X\times PX,\loops X\times\loops X)\to(PX,\loops X)$.

Let $u$ be the class with $\partial u=e'$. The cross product
$e\times u\in H_*(\loops X\times PX,\loops X\times\loops X)$
has $\partial(e\times u)=(-1)^{|e|}e\times e'$, $e$ being an
absolute class, so
$\partial\,a_*(e\times u)=(-1)^{|e|}a_*(e\times e')=(-1)^{|e|}ee'$.
By the description of $\sigma$ above,
$\sigma(ee')=(-1)^{|e|}\,q_*a_*(e\times u)$.

Now $q_*a_*=q_*\mathrm{pr}_{2*}$, and $\mathrm{pr}_2$ collapses
the first factor to a point, so
$\mathrm{pr}_{2*}(e\times u)=\varepsilon(e)\,u$. Since $|e|>0$,
$\varepsilon(e)=0$, and hence $\sigma(ee')=0$.
\end{proof}

\begin{lemma}[Distinct exponents]\label{lem:indec}
Let $(K,R)$ be an exterior pair with $K$ simply connected,
$H_*(K;R)=\Lambda_R(a_1,\dots,a_r)$, and the exponents of $K$
pairwise distinct. Then for any polynomial generators
$b_1,\dots,b_r$ of $H_*(\loops K;R)$, after reindexing,
$\sigma(b_k)=u_k\,a_k$ with $u_k\in R^\times$.
\end{lemma}

\begin{proof}
By Proposition~\ref{prop:frame} and Lemma~\ref{lem:edge},
$H_*(\loops K;R)=R[y_1,\dots,y_r]$ with $\sigma(y_k)=a_k$. The
degrees are distinct, so after reindexing $|b_k|=|y_k|$ and
$b_k=u_ky_k+{}$decomposable, with $u_k$ a unit since the $b_k$
generate. The suspension $\sigma$ kills decomposables (Lemma~\ref{lem:sigmaprops}),
so $\sigma(b_k)=u_ka_k$.
\end{proof}

\noindent Lemma~\ref{lem:indec} is used for $\Sp(n)$ in
Section~\ref{sec:hopf}. In fact the image of $\sigma$ lies in
the primitives of $H_*(X;R)$ for every based space $X$
\cite[Lemma~8]{Hepworth}.

\section{The BV algebra of \texorpdfstring{$L\Sp(n)$}{LSp(n)}}
\label{sec:hopf}

In this section we compute the BV algebra of $L\Sp(n)$ over
$\Z$ (Theorem~\ref{thmC}). We write out Hepworth's operators of
Definition~\ref{def:hepworth2} in the polynomial generators of
Kono and Kozima (Theorem~\ref{thm:kk}), with the homology
suspension supplied by Section~\ref{sec:suspension}.
Remark~\ref{rem:kupers} compares the result with Kupers'
computation \cite{Kupers}.

The Pontryagin ring $H_*(\loops\Sp(n);\Z)$, the loop factor of
Theorem~\ref{thm:hepworth1}, we take from Kono and Kozima.

\begin{theorem}[{\cite[Theorem~4.3]{KonoKozima}}]\label{thm:kk}
The Pontryagin ring
$H_*(\loops\Sp(n);\Z)\cong\Z[z_1,z_3,\dots,z_{2n-1}]$ is a
polynomial ring on generators $z_{2k-1}$ of degree $4k-2$, with
coproduct
\[
\psi(z_r)\;=\;z_r\otimes1+1\otimes z_r
+2\!\!\sum_{\substack{i+j=r\\ i,j\geq1}}\!\!z_i\otimes z_j
\qquad(r\ \text{odd}),
\]
where the classes of even index are the universal integral
polynomials in the odd generators given by the recursion
\[
z_{2m}\;=\;2\sum_{i=1}^{m-1}(-1)^{i+1}z_i\,z_{2m-i}
+(-1)^{m+1}z_m^{2},
\]
with the convention $z_{2m}=0$ for $m\leq0$.
\end{theorem}

\noindent From the fibration $\Sp(n{-}1)\to\Sp(n)\to S^{4n-1}$
and the associated Serre spectral sequence, $H_*(\Sp(n);\Z)$ is free with
Poincar\'e polynomial $\prod_{k=1}^{n}(1+t^{4k-1})$; so by
Theorem~\ref{thm:whichpairs} $(\Sp(n),\Z)$ is an exterior pair,
with one odd primitive generator in each degree $4k-1$,
$1\leq k\leq n$. The exponents of $\Sp(n)$, being
$\{1,3,\dots,2n-1\}$, are pairwise distinct. We now write
these primitive generators of $H_*(\Sp(n);\Z)$ as suspensions
of the polynomial generators of Theorem~\ref{thm:kk}. Since the
exponents are distinct, Lemma~\ref{lem:indec} applies to the
$z_{2k-1}$; and every term of the recursion of
Theorem~\ref{thm:kk} is a product of two classes of positive
degree, so $z_{2j}$ is decomposable and $\sigma(z_{2j})=0$ by
Lemma~\ref{lem:sigmaprops}.

\begin{proposition}[The homology suspension]\label{prop:susp}
The classes $\sigma(z_1),\sigma(z_3),\dots,\sigma(z_{2n-1})$
are a basis of the odd primitives of $H_*(\Sp(n);\Z)$, and
$\sigma(z_{2j})=0$.
\end{proposition}

\noindent Write $a_{4k-1}:=\sigma(z_{2k-1})$, indexing the
exterior generators of Definition~\ref{def:exterior} by degree,
and $x_{4k-1}$ for the dual basis
(Definition~\ref{rem:symbols}), so that
\[
H_*(\Sp(n);\Z)=\Lambda_\Z(a_3,a_7,\dots,a_{4n-1}),
\qquad |a_{4k-1}|=4k-1 .
\]
Orient $\Sp(n)$ by $[\Sp(n)]=a_3a_7\cdots a_{4n-1}$. The
intersection classes of Definition~\ref{def:alpha} are
\[
\alpha_{4k-1}\;=\;(-1)^{k-1}a_3\cdots\widehat{a_{4k-1}}\cdots a_{4n-1},
\qquad |\alpha_{4k-1}|=-(4k-1),
\]
with $a_{4k-1}\alpha_{4k-1}=[\Sp(n)]$, and
$\HH_*(\Sp(n);\Z)=\Lambda_\Z(\alpha_3,\alpha_7,\dots,\alpha_{4n-1})$.
For the primitive basis $p_k=a_{4k-1}$ of
Definition~\ref{def:hepworth2}, Lemma~\ref{lem:intersection}
gives $\delta_k=\partial/\partial\alpha_{4k-1}$.

\begin{proposition}[The based loop side]\label{prop:partial}
With respect to the presentation of Theorem~\ref{thm:kk}, the
derivations $\partial_k$ of Definition~\ref{def:hepworth2} for
the basis $p_k=a_{4k-1}$ are
\[
\partial_k\;=\;\frac{\partial}{\partial z_{2k-1}}
\;+\;2\sum_{l=k+1}^{n}z_{2(l-k)}\,
\frac{\partial}{\partial z_{2l-1}},
\qquad\text{that is,}\qquad
\partial_k(z_{2l-1})=\delta_{kl}+2\,z_{2(l-k)},
\]
with $z_{2m}=0$ for $m\leq0$.
\end{proposition}

\begin{proof}
By Proposition~\ref{prop:reduce}, $\partial_k$ is the
derivation with
$\partial_k(z_{2l-1})=\sum\langle x_{4k-1},\sigma(z_{2l-1,(1)})\rangle\,z_{2l-1,(2)}$.
The coproduct of Theorem~\ref{thm:kk} has the terms
$z_{2l-1}\otimes1$, $1\otimes z_{2l-1}$, and $2z_i\otimes z_j$
with $i+j=2l-1$ and $i,j\geq1$. By Proposition~\ref{prop:susp},
$\langle x_{4k-1},\sigma(z_i)\rangle=\delta_{i,2k-1}$ for
$i\geq1$, and $\sigma(1)=0$. The three kinds of terms therefore
contribute $\delta_{kl}$, $0$, and $2z_{2l-1-(2k-1)}=2z_{2(l-k)}$,
the last only when $l>k$.
\end{proof}

\begin{theorem}\label{thmC}
There is an isomorphism of graded algebras
\[
\HH_*(L\Sp(n);\Z)\;\cong\;
\Z[z_1,z_3,\dots,z_{2n-1}]\otimes
\Lambda_{\Z}(\alpha_3,\alpha_7,\dots,\alpha_{4n-1}),
\]
with $|z_{2k-1}|=4k-2$ and $|\alpha_{4k-1}|=-(4k-1)$, under which
\[
\Delta\;=\;\sum_{k=1}^{n}\ \partial_k\otimes
\frac{\partial}{\partial\alpha_{4k-1}},
\qquad
\partial_k\;=\;\frac{\partial}{\partial z_{2k-1}}
\;+\;2\sum_{l=k+1}^{n}z_{2(l-k)}\,
\frac{\partial}{\partial z_{2l-1}}\,,
\]
and the induced Gerstenhaber bracket of degree $1$ is
determined by
\[
\{z_{2l-1},\alpha_{4k-1}\}\;=\;\delta_{kl}+2\,z_{2(l-k)},
\qquad
\{z,z'\}\;=\;\{\alpha,\alpha'\}\;=\;0 ,
\]
the bracket being a derivation in each variable.
\end{theorem}

\begin{proof}
By Theorem~\ref{thm:hepworth1},
$\HH_*(L\Sp(n);\Z)\cong H_*(\loops\Sp(n);\Z)\otimes\HH_*(\Sp(n);\Z)$
as graded algebras, which is the polynomial ring of
Theorem~\ref{thm:kk} tensored with the exterior ring above.

The homology $H_*(\Sp(n);\Z)$ is torsion-free, so
Proposition~\ref{prop:hepworth3} applies over $\Z$:
$\Delta=\sum_k\partial_k\otimes\delta_k$ for the basis
$p_k=a_{4k-1}$, with $\delta_k=\partial/\partial\alpha_{4k-1}$
as above and $\partial_k$ as in Proposition~\ref{prop:partial}.

For the bracket, in the convention
$\Delta(ab)=\Delta(a)b+(-1)^{|a|}a\Delta(b)+(-1)^{|a|}\{a,b\}$,
$\Delta$ vanishes on $\Z[z]\otimes1$, which carries no
$\alpha$, and on the $\alpha_{4k-1}$, as $\partial_k(1)=0$.
Hence
\[
\{z_{2l-1},\alpha_{4k-1}\}
=\Delta\big(z_{2l-1}\alpha_{4k-1}\big)
=\partial_k(z_{2l-1})
=\delta_{kl}+2\,z_{2(l-k)},
\]
and likewise $\{z,z'\}=\Delta(zz')=0$ and
$\{\alpha,\alpha'\}=\Delta(\alpha\alpha')=0$.
\end{proof}

\begin{remark}[Comparison with Kupers]\label{rem:kupers}
Kupers \cite[Theorem~4.1]{Kupers} obtains the graded ring of
Theorem~\ref{thmC} and writes $\Delta$ as a sum of tensor
products of derivations in the same generators. He takes the
exterior generators $a_{4k-1}$ of $H_*(\Sp(n);\Z)$ from Yokota's
cell structure and the polynomial generators $z_{2k-1}$ from
Theorem~\ref{thm:kk}, and computes the homology suspension
between them by naturality, through the quaternionic Stiefel
manifold $\Sp(n)/\Sp(n{-}1)=S^{4n-1}$ and the inclusions
$\Sp(k)\to\Sp(n)$, instead of Lemma~\ref{lem:indec}. The two
bases are chosen independently, so the sign
$\sigma(z_{2k-1})=\pm a_{4k-1}$ is left open there; here
$a_{4k-1}:=\sigma(z_{2k-1})$, and there is no sign to determine.
His even-index classes agree with the $z_{2m}$ of
Theorem~\ref{thm:kk}. In his statement the $k$-th summand sends
$z_{2l-1}\otimes\alpha_{4k-1}$ to $z_{2(l-k)}\otimes1$ for
$k<l$, whereas the coproduct of Theorem~\ref{thm:kk}, which he
also quotes \cite[Theorem~4.4]{Kupers}, gives $2z_{2(l-k)}$
(Proposition~\ref{prop:partial}). Over $\F_2$ the two statements
differ: the mixing term of Theorem~\ref{thmC} vanishes. The
bracket is not considered there.
\end{remark}

\begin{remark}[Degenerations]\label{rem:degen}
The primitives of $H_{4k-2}(\loops\Sp(n);\Z)$ form a free
module of rank one, on a class
$\rho_k\equiv(2k-1)\,z_{2k-1}$ modulo decomposables
\cite[\S5]{KonoKozima}. Once the integers $1,3,\dots,2n-1$ are
inverted, over $\Z[1/(2n-1)!]$ or $p$-locally for $p>2n-1$, the
classes $(2k-1)^{-1}\rho_k$ are polynomial generators, they are
primitive, and $\Delta$ is a pure tensor
(Definition~\ref{def:mixing}). The next section shows that no
localization is needed.
\end{remark}

\section{BV rigidity}\label{sec:rigidity}

In this section we assume $(G,R)$ is an exterior pair with $\pi_1(G)$ torsion-free. By
Theorem~\ref{thm:hepworth1},
\[
\HH_*(LG;R)\cong H_*(\loops G;R)\otimes\HH_*(G;R),
\qquad
\Delta(a\otimes x)=\sum a_{(1)}\otimes\sigma(a_{(2)})\,x .
\]
The intersection ring and the Pontryagin action have the same form
for every exterior pair, an exterior algebra with its partial
derivatives (Section~\ref{sub:intersection}), and the
Pontryagin ring and the homology suspension come from the cobar
construction and its edge map (Sections~\ref{sub:frame}
and~\ref{sec:suspension}). Only the coproduct of $H_*(\loops G;R)$ depends on $G$. For
$\Sp(n)$, $n\geq2$, it is not primitively generated, and this is the
source of the terms $2z_{2(l-k)}$ in the operator of
Theorem~\ref{thmC}. We show that it
does not affect the BV algebra. A change of the exterior generators,
written out for $\Sp(n)$ in Example~\ref{ex:spshear} and in general
in Theorem~\ref{thm:rigid}, carries $\HH_*(LG;R)$ onto the BV
algebra of $LX_G$, for $X_G$ the product of odd spheres with the
exponents of $G$, showing that the BV algebra of such $G$ is rigid.
Over an integral domain we will see that the graded algebra $\HH_*(LG;R)$ already determines the exponents, hence the BV algebra (Corollary~\ref{cor:classify}). When
$G$ is simply connected, we show that the BV algebra is the Hochschild cohomology of
$H^*(G;R)$ (Theorem~\ref{thm:hochmodel}). We end by conjecturing that this identification of Hochschild cohomology is induced by an $A_\infty$ quasi-isomorphism compatible with Poincar\'e duality (Conjecture~\ref{conj:pdformal}).

\subsection{The associated product of odd spheres}\label{sub:spheres}

Let $(G,R)$ be an exterior pair with exponents $m_1,\dots,m_r$
(Definition~\ref{def:exterior}), and let
\[
X_G\;=\;\prod_{i=1}^{r}S^{2m_i+1},
\]
an exponent $0$ giving a circle. The exponents are the classical
ones (Theorem~\ref{thm:whichpairs}), so $X_G$ is the product of
Section~\ref{sub:exterior}: it is rationally equivalent to $G$ and
has the homology of $G$ over $R$.

\begin{lemma}[Products]\label{lem:product}
Let $M$, $N$ be closed oriented manifolds with $H_*(LM;R)$ a
free $R$-module. Then the K\"unneth map is an isomorphism of
Batalin-Vilkovisky algebras
\[
\HH_*(LM;R)\otimes\HH_*(LN;R)\;\cong\;\HH_*(L(M\times N);R),
\qquad
\Delta_{M\times N}\;=\;\Delta_M\otimes1+1\otimes\Delta_N .
\]
\end{lemma}

\begin{proof}
Since $L(M\times N)=LM\times LN$ and $H_*(LM;R)$ is free, the
K\"unneth map is an isomorphism of graded modules.

It is multiplicative. Let $m=\dim M$ and $n=\dim N$. The loop
product of $M$ is the composite
\[
H_*(LM\times LM)\xrightarrow{\ \tau\ }H_{*-m}(LM\times_MLM)
\xrightarrow{\ \mathrm{concat}_*\ }H_{*-m}(LM),
\]
where $LM\times_MLM=\mathrm{ev}^{-1}(\Delta_M)$ is the space of
composable pairs, for
$\mathrm{ev}\colon LM\times LM\to M\times M$,
$(\alpha,\beta)\mapsto(\alpha(0),\beta(0))$, and
$\tau=(-)\cap u_M$ is the Thom collapse: $\mathrm{ev}$ is a
fibration, so a tubular neighbourhood of $\Delta_M$ pulls back
to one of $LM\times_MLM$, the total space of
$\mathrm{ev}^*\nu_M$ for the normal bundle $\nu_M$ of the
diagonal, and $u_M$ is the pullback of its Thom class.

A loop in $M\times N$ is a pair of loops, $L(M\times N)=LM\times LN$.
Swapping the middle factors of $(LM\times LN)\times(LM\times LN)$
identifies it with $(LM\times LM)\times(LN\times LN)$,
$\mathrm{ev}_{M\times N}$ with $\mathrm{ev}_M\times\mathrm{ev}_N$
and $\Delta_{M\times N}$ with $\Delta_M\times\Delta_N$. Hence the
composable pairs of $M\times N$ are
$(LM\times_MLM)\times(LN\times_NLN)$; the normal bundle of
$\Delta_M\times\Delta_N$ is $\nu_M\times\nu_N$, with Thom class
$u_M\times u_N$; and concatenation in $M\times N$ is
componentwise, $\mathrm{concat}_M\times\mathrm{concat}_N$. For
$a,a'\in H_*(LM)$ and $b,b'\in H_*(LN)$,
\begin{align*}
(a\times b)\cdot(a'\times b')
&=\mathrm{concat}_*\big((a\times b\times a'\times b')\cap u_{M\times N}\big)\\
&=\pm\,(\mathrm{concat}_M\times\mathrm{concat}_N)_*
\big(((a\times a')\cap u_M)\times((b\times b')\cap u_N)\big)\\
&=\pm\,(a\cdot a')\times(b\cdot b'),
\end{align*}
the signs Koszul signs, from the swap and from the cap product
of cross products. So the K\"unneth map is multiplicative for
the tensor product of graded algebras.

For the BV operators, $\Delta_M(a)=\rho_{M*}([S^1]\times a)$ for the
rotation $\rho_M\colon S^1\times LM\to LM$. The rotation of
$M\times N$ is $(t,\alpha,\beta)\mapsto(\rho_M(t,\alpha),\rho_N(t,\beta))$,
which factors through the diagonal $D$ of $S^1$, and
$D_*[S^1]=[S^1]\times1+1\times[S^1]$; so
\[
\Delta_{M\times N}(a\times b)
=\Delta_M(a)\times b+(-1)^{|a|}a\times\Delta_N(b),
\]
that is, $\Delta_{M\times N}=\Delta_M\otimes1+1\otimes\Delta_N$.
\end{proof}

\begin{proposition}[Products of odd spheres]\label{prop:spheres}
For $X=S^{2m_1+1}\times\cdots\times S^{2m_r+1}$ with $m_k\geq1$,
\begin{gather*}
\HH_*(LX;\Z)\;\cong\;\Z[u_{2m_1},\dots,u_{2m_r}]\otimes
\Lambda_{\Z}(d_{2m_1+1},\dots,d_{2m_r+1}),\\
\Delta_X\;=\;\sum_{k=1}^{r}
\frac{\partial}{\partial u_{2m_k}}\otimes
\frac{\partial}{\partial d_{2m_k+1}}.
\end{gather*}
\end{proposition}

\begin{proof}
For one odd sphere,
$\HH_*(LS^{2m+1};\Z)\cong\Z[u_{2m}]\otimes\Lambda_\Z(d_{2m+1})$
with $\Delta=\partial/\partial u_{2m}\otimes\partial/\partial d_{2m+1}$
\cite[Theorem~16]{Menichi}. Lemma~\ref{lem:product}, applied
factor by factor, gives the tensor product of these algebras
with $\Delta_X=\sum_k\id\otimes\cdots\otimes\Delta_k\otimes
\cdots\otimes\id$.
\end{proof}

\noindent The homology of $LX$ is free, so reduction of
coefficients carries the presentation and the BV operator to any
$R$: $\HH_*(LX;R)\cong R[u]\otimes\Lambda_R(d)$ with the same
BV operator. Each summand of $\Delta_X$ touches a single pair of
generators; we call an BV operator of this form a \emph{pure
tensor}.

\subsection{Mixing}\label{sub:mixing}

Let $K$ be the compact factor of the universal cover of $G$
(Section~\ref{sub:cover}), so that $(K,R)$ is an exterior pair
(Corollary~\ref{cor:Kexterior}). By Proposition~\ref{prop:frame},
$H_*(\loops K;R)$ is a
polynomial ring on the $y_k$, and by Lemma~\ref{lem:edge} the
suspension carries the $y_k$ bijectively onto the basis $a_k$
of the odd primitives. For a presentation of this kind,
Proposition~\ref{prop:reduce} reduces the BV operator of
Theorem~\ref{thm:hepworth1} to a pure tensor plus correction
terms. The correction terms are what we call the mixing of the
presentation.

\begin{definition}[Mixing]\label{def:mixing}
Let $(G,R)$ be an exterior pair and let $H_*(\loops G;R)$ be a
polynomial ring on generators $b_1,\dots,b_r$ whose homology
suspensions $p_k=\sigma(b_k)$ are a basis of the odd primitives
of $H_*(G;R)$. Let $\partial_k$, $\delta_k$ be the operators of
Definition~\ref{def:hepworth2} for this basis. The
\emph{mixing terms} of the presentation $\{b_k\}$ are the
classes
\[
\gamma_{kl}\;:=\;\partial_k(b_l)-\delta_{kl}
\;\in\;H_*(\loops G;R),
\]
so that
\[
\Delta\;=\;\sum_{k}\frac{\partial}{\partial b_k}\otimes\delta_k
\;+\;\sum_{k,l}\gamma_{kl}\,
\frac{\partial}{\partial b_l}\otimes\delta_k
\]
by Proposition~\ref{prop:reduce}. The presentation is
\emph{unmixed} if every $\gamma_{kl}$ vanishes, that is, if
$\Delta$ is a pure tensor.
\end{definition}

\noindent The $\gamma_{kl}$ are triangular: $\gamma_{kl}$ has
degree $|b_l|-|b_k|$, and for $|b_l|=|b_k|$ the degree-zero part
of $\psi(b_l)$ is $b_l\otimes1$, so
$\partial_k(b_l)=\langle p^k,\sigma(b_l)\rangle=\delta_{kl}$;
hence $\gamma_{kl}=0$ unless $|b_l|>|b_k|$. Primitive
generators are unmixed, $\psi(b_l)=b_l\otimes1+1\otimes b_l$
giving $\partial_k(b_l)=\delta_{kl}$. The presentation of
Theorem~\ref{thmC} is mixed: $\gamma_{kl}=2\,z_{2(l-k)}$ for
$l>k$.

\subsection{The shear}\label{sub:shear}

The mixing is removed by a change of the odd generators, over
any coefficients.

\begin{proposition}[The shear]\label{thm:shear}
Let $R$ be a commutative ring and let
\[
\HH\;=\;R[b_1,\dots,b_r]\otimes\Lambda_R(\alpha_1,\dots,\alpha_r),
\]
with the $b_k$ of positive even degree and the $\alpha_k$ of
odd degree $|\alpha_k|=-|b_k|-1$, be a Batalin-Vilkovisky algebra whose BV operator has
the form
\[
\Delta\;=\;\sum_{k=1}^{r}\partial_k\otimes
\frac{\partial}{\partial\alpha_k},
\qquad
\partial_k(b_l)\;=\;\delta_{kl}+\gamma_{kl},
\]
with each $\partial_k$ a derivation of the polynomial factor and
each $\gamma_{kl}$ a polynomial of degree $|b_l|-|b_k|$, zero
unless $|b_l|>|b_k|$. Then
\[
\Phi(b_k)\;=\;u_k,
\qquad
\Phi(\alpha_k)\;=\;d_k+\sum_{l}\gamma_{kl}(u)\,d_l,
\]
with $\gamma_{kl}(u)$ the polynomial $\gamma_{kl}$ evaluated at
$b_i=u_i$, defines an isomorphism of Batalin-Vilkovisky
algebras onto the pure tensor
\[
R[u_1,\dots,u_r]\otimes\Lambda_R(d_1,\dots,d_r),
\qquad
\Delta_X\;=\;\sum_{k=1}^{r}\frac{\partial}{\partial u_k}\otimes
\frac{\partial}{\partial d_k},
\]
with $|u_k|=|b_k|$ and $|d_k|=|\alpha_k|$.
\end{proposition}

\begin{proof}
Write $c_{km}=\delta_{km}+\gamma_{km}(u)$, so
$\Phi(\alpha_k)=\sum_m c_{km}d_m$; the $c_{km}$ are of even degree.

The source is the free graded-commutative algebra on the even $b_k$
and the odd $\alpha_k$. In the target, the $\Phi(\alpha_k)$ are odd,
square to zero and anticommute, hence $\Phi$ extends to a ring map.
Since restriction of $\Phi$ to the polynomial factor is the ring
isomorphism $R[b]\to R[u]$, $b_i\mapsto u_i$, we see both source and target as modules over the polynomial ring $R[u]$ and $\Phi$ is $R[u]$-linear.
On the free $R[u]$-module spanned by $d_1,\dots,d_r$,
$\Phi$ is given by the matrix $(c_{km})$. Ordering the generators by
degree, this matrix is unitriangular, because $\gamma_{km}=0$ unless
$|b_m|>|b_k|$, and $c_{kk}=1$. Therefore $\Phi$ is a ring isomorphism.

By Proposition~\ref{prop:spheres} and the remark following it,
$R[u]\otimes\Lambda_R(d)$ with $\Delta_X$ is the loop homology
$\HH_*(LX;R)$ of the product of odd spheres
$X=\prod_{k}S^{|b_k|+1}$, so $\Delta_X$ is a BV operator.
Define $\Delta':=\Phi^{-1}\Delta_X\Phi$, the BV operator $\Delta_X$
transported to $\HH$ along the ring isomorphism $\Phi$. We will
prove that $\Delta=\Delta'$ on the generators and on their
products of two, and by the seven-term relation the same then
holds on the whole ring $\HH$. The elements to check are $1$,
$b_k$, $\alpha_k$, $b_lb_m$, $b_l\alpha_k$ and $\alpha_k\alpha_j$.

On the target, for a polynomial $P$ in the $u$'s,
\[
\Delta_X(P\,d_m)=\frac{\partial P}{\partial u_m},
\qquad
\Delta_X(P\,d_md_n)=\frac{\partial P}{\partial u_m}\,d_n
-\frac{\partial P}{\partial u_n}\,d_m,
\]
and $\Delta_X$ vanishes on $R[u]$. On the source, $\Delta$
vanishes on $R[b]$, and $\Delta(\alpha_k)=\partial_k(1)=0$,
$\Delta(b_l\alpha_k)=\partial_k(b_l)$,
$\Delta(\alpha_k\alpha_j)=0$. The polynomial $\gamma_{km}$ has
degree $|b_m|-|b_k|<|b_m|$, so it does not involve $u_m$, and
$\sum_m\partial c_{km}/\partial u_m=0$.

Since $\Phi$ take polynomial generators to polynomial, both the BV operators vanish on $1$, $b_k$ and $b_lb_m$. On $\alpha_k$
and $b_l\alpha_k$,
\[
\Delta_X\Phi(\alpha_k)=\sum_m\frac{\partial c_{km}}{\partial u_m}=0,
\qquad
\Delta_X\Phi(b_l\alpha_k)
=\sum_m\frac{\partial(u_lc_{km})}{\partial u_m}
=c_{kl}=\Phi\big(\partial_k(b_l)\big),
\]
hence $\Delta$ and $\Delta'$ matches on the generators $\alpha_k$ and the product $b_l\alpha_k$. Finally, for $\alpha_k\alpha_j,\ k\neq j$, we apply $\Delta$ twice to
$b_n\alpha_k\alpha_j$ giving us $[\partial_j,\partial_k](b_n)$. So
$\Delta^2=0$ says the $\partial_k$ commute:
$\partial_k(\gamma_{jn})=\partial_j(\gamma_{kn})$.
Since $\Phi(\alpha_k)\Phi(\alpha_j)=\sum_{m,n}c_{km}c_{jn}\,d_md_n$,
the formula for $\Delta_X(P\,d_md_n)$ gives
\[
\Delta_X\big(\Phi(\alpha_k)\Phi(\alpha_j)\big)
=\sum_{m,n}\frac{\partial(c_{km}c_{jn})}{\partial u_m}\,d_n
-\sum_{m,n}\frac{\partial(c_{km}c_{jn})}{\partial u_n}\,d_m .
\]
We expand the two sums by the Leibniz rule. In the first sum,
\begin{align*}
\sum_{m,n}\frac{\partial c_{km}}{\partial u_m}\,c_{jn}\,d_n
&=\Big(\sum_m\frac{\partial c_{km}}{\partial u_m}\Big)\Phi(\alpha_j)=0,\\
\sum_{m,n}c_{km}\frac{\partial c_{jn}}{\partial u_m}\,d_n
&=\sum_n\big(\partial_k\gamma_{jn}\big)(u)\,d_n .
\end{align*}
The first equality is $\sum_m\partial c_{km}/\partial u_m=0$. The
second is the chain rule
$\partial_k f=\sum_m\partial_k(b_m)\,\partial f/\partial b_m$ for
$f=\gamma_{jn}$, with $\Phi(\partial_k b_m)=c_{km}$. The second sum
equals $\sum_m\partial_j(\gamma_{km})(u)\,d_m$ in the same way.
Hence
\[
\Delta_X\big(\Phi(\alpha_k)\Phi(\alpha_j)\big)
=\sum_n\big(\partial_k(\gamma_{jn})-\partial_j(\gamma_{kn})\big)(u)\,d_n ,
\]
and each coefficient is zero because the $\partial_k$ commute.
Hence $\Delta'(\alpha_k\alpha_j)=0=\Delta(\alpha_k\alpha_j)$.

The two operators agree on all the generators and their products, so
$\Delta'=\Delta$, and $\Phi$ is an isomorphism of BV algebras.
\end{proof}

\begin{example}[The symplectic group]\label{ex:spshear}
For the presentation of Theorem~\ref{thmC}, $\partial_k$ is a
derivation with $\partial_k(z_{2l-1})=\delta_{kl}+2z_{2(l-k)}$
(Proposition~\ref{prop:partial}), and $z_{2(l-k)}$ is a polynomial
of degree $4(l-k)=|z_{2l-1}|-|z_{2k-1}|$, zero unless $l>k$. So
Proposition~\ref{thm:shear} applies with $b_k=z_{2k-1}$,
$\alpha_k=\alpha_{4k-1}$ and $\gamma_{kl}=2z_{2(l-k)}$, and gives
an isomorphism of Batalin-Vilkovisky algebras
\[
\Phi\colon\HH_*(L\Sp(n);\Z)\longrightarrow
\HH_*\big(L(S^3\times S^7\times\cdots\times S^{4n-1});\Z\big),
\]
\[
\Phi(z_{2k-1})=u_{4k-2},
\qquad
\Phi(\alpha_{4k-1})=d_{4k-1}+2\sum_{l=k+1}^{n}z_{2(l-k)}(u)\,d_{4l-1},
\]
with $z_{2(l-k)}(u)$ the polynomial $z_{2(l-k)}$ of
Theorem~\ref{thm:kk} evaluated at $z_{2i-1}=u_{4i-2}$, onto the
pure tensor of Proposition~\ref{prop:spheres}. For $n=2$,
$\Phi(\alpha_3)=d_3+2u_2^2\,d_7$ and $\Phi(\alpha_7)=d_7$. So the
mixing term of Theorem~\ref{thmC} is removed over $\Z$, without
the localization of Remark~\ref{rem:degen}.
\end{example}

\subsection{Rigidity}\label{sub:rigid}

Write $G\cong T^k\times K$ with $K$ simply connected
(Theorem~\ref{thm:borel32}). The torus factor splits off from the loop homology of $G$ by Lemma~\ref{lem:product}, and for $K$ the argument of
Example~\ref{ex:spshear} applies verbatim: the mixing terms are
triangular (Section~\ref{sub:mixing}), so the shear of
Proposition~\ref{thm:shear} carries $\HH_*(LK;R)$ onto
$\HH_*(LX_K;R)$.

\begin{theorem}[BV rigidity for exterior pairs]\label{thm:rigid}
Let $(G,R)$ be an exterior pair with $\pi_1(G)$ torsion-free,
that is, with $G_{ss}$ simply connected. Then there is an
isomorphism of Batalin-Vilkovisky algebras
\[
\HH_*(LG;R)\;\cong\;\HH_*(LX_G;R).
\]
\end{theorem}

\begin{proof}
By Theorem~\ref{thm:borel32}, $G\cong T^k\times K$ with $K=G_{ss}$
simply connected and $(K,R)$ exterior
(Corollary~\ref{cor:Kexterior}), and $X_G=T^k\times X_K$. By
Lemma~\ref{lem:product} both sides have $\HH_*(LT^k;R)$ tensored
with the $K$-side, with the same BV operator on the torus factor,
so it suffices to treat $K$.

Let $H_*(K;R)=\Lambda_R(a_1,\dots,a_r)$ with $|a_k|=2m_k+1$,
$m_k\geq1$. Then $H_*(\loops K;R)=R[y_1,\dots,y_r]$ with
$|y_k|=2m_k$ and $\sigma(y_k)=a_k$ (Proposition~\ref{prop:frame},
Lemma~\ref{lem:edge}), $\HH_*(K;R)=\Lambda_R(\alpha_1,\dots,\alpha_r)$
with $\delta_k=\partial/\partial\alpha_k$
(Lemma~\ref{lem:intersection}), and
$\Delta_K=\sum_k\partial_k\otimes\partial/\partial\alpha_k$
(Proposition~\ref{prop:reduce}). The mixing terms are
triangular (Section~\ref{sub:mixing}), so
Proposition~\ref{thm:shear} applies, and
its target, with $|u_k|=2m_k$ and $|d_k|=-(2m_k+1)$, is
$\HH_*(LX_K;R)$ (Proposition~\ref{prop:spheres}).
\end{proof}

\begin{corollary}[Classification]\label{cor:classify}
Let $R$ be an integral domain and $(G,R)$, $(G',R)$ exterior
pairs with $\pi_1(G)$ and $\pi_1(G')$ torsion-free. The
following are equivalent:
\begin{enumerate}
\item $\HH_*(LG;R)\cong\HH_*(LG';R)$ as Batalin-Vilkovisky
algebras;
\item $\HH_*(LG;R)\cong\HH_*(LG';R)$ as graded algebras;
\item $G$ and $G'$ have the same multiset of $R$-exponents;
\item $G$ and $G'$ are rationally homotopy equivalent.
\end{enumerate}
\end{corollary}

\begin{proof}
(3)$\Rightarrow$(1): both algebras are BV-isomorphic to
$\HH_*(LX;R)$ for the common $X=X_G=X_{G'}$
(Theorem~\ref{thm:rigid}). (1)$\Rightarrow$(2) is trivial.

(2)$\Rightarrow$(3). It is enough to recover the exponents of
$G$ from the graded algebra $A=\HH_*(LG;R)$. Let $k$ be the
multiplicity of the exponent $0$ and $\{m_1,\dots,m_r\}$ the
positive exponents. By Theorem~\ref{thm:rigid},
Lemma~\ref{lem:product} and Proposition~\ref{prop:spheres},
\[
A\;\cong\;R[t_1^{\pm1},\dots,t_k^{\pm1}]\otimes
\Lambda_R(\theta_1,\dots,\theta_k)\otimes
R[b_1,\dots,b_r]\otimes\Lambda_R(\alpha_1,\dots,\alpha_r),
\]
the first two factors being $\HH_*(LT^k;R)$, with
\[
|t_i|=0,\qquad|\theta_i|=-1,\qquad
|b_j|=2m_j,\qquad|\alpha_j|=-(2m_j+1).
\]
The odd generators $\theta_i$, $\alpha_j$ anticommute and
square to zero, so the ideal $N$ they generate is nilpotent.
The quotient $A/N=R[t_1^{\pm1},\dots,t_k^{\pm1}][b_1,\dots,b_r]$
is an integral domain, $R$ being one. Hence $N$ is the
nilradical of $A$, and $\bar A:=A/N$ depends only on the ring
$A$. The degree-zero part of $\bar A$ is the
Laurent polynomial ring $R[t_1^{\pm1},\dots,t_k^{\pm1}]$, whose
base change to the fraction field $F$ of $R$ is
$F[t_1^{\pm1},\dots,t_k^{\pm1}]$, of Krull dimension $k$; this
gives $k$. In positive degrees
$\bar A$ is a polynomial ring over its degree-zero part, whose
indecomposables are free with basis in degrees
$2m_1,\dots,2m_r$; this gives the positive exponents.

(3)$\Leftrightarrow$(4): the $R$-exponents are the classical
exponents (Theorem~\ref{thm:whichpairs}), and the map
$X_G\to G$ of Section~\ref{sub:exterior} is a rational
homotopy equivalence, so the exponents and the rational
homotopy type determine one another.
\end{proof}

\begin{remark}[Not a map of spaces]\label{rem:sharp}
Consider an exterior pair $(K,\Z)$ with $K$ simply connected,
and suppose there is a based map $f$ between $X_K$ and $K$, in
either direction, such that $Lf_*$ is an isomorphism on integral
homology. As the constant loops $c$ split the evaluation, we have
$\mathrm{ev}\circ c=\mathrm{id}$, and $Lf\circ c=c\circ f$, so
likewise $\mathrm{ev}\circ Lf=f\circ\mathrm{ev}$. Therefore $f_*$
is a retract of $Lf_*$, which means it is an isomorphism, and
thus $f$ is a homotopy equivalence of two simply connected finite
complexes. Thus $\loops f$ is an equivalence of $H$-spaces, and
$H_*(\loops X_K)\cong H_*(\loops K)$ as Hopf algebras. However,
$H_*(\loops X_K)$ is primitively generated, while $H_*(\loops K)$
is not: for $\Sp(n)$, $n\geq2$, its primitive in degree $6$ is
$3z_3-2z_1^3$ (Remark~\ref{rem:degen}), and for $\SU(n)$,
$n\geq3$, its primitive in degree $4$ is $2e_4-e_2^2$, by
\cite[Theorem~2-1]{Tamanoi}. Furthermore, already
$\pi_6(S^3\times S^7)$ and $\pi_6(S^3\times S^5)$ contain
$\pi_6(S^3)$, of order $12$ \cite[Chapitre~IV, Proposition~10]{Serre53},
whereas $\pi_6(\Sp(2))=0$ and $\pi_6(\SU(3))$ has order $6$
\cite[Chapter~IV, Table~4.2]{MimuraToda}. Hence for these groups
the isomorphism of Theorem~\ref{thm:rigid} cannot be given by a
map of spaces.
\end{remark}

\begin{remark}[Rational coefficients]\label{rem:rational}
Over $\Q$, the isomorphism of Theorem~\ref{thm:rigid} is given by
formality. Let $G\cong T^k\times K$, where $K$ is simply
connected. Both $K$ and $X_K$ are rationally formal, since their
cohomology is an exterior algebra on odd-dimensional generators,
and F\'elix and Thomas construct the rational BV algebra of the
free loop space of a simply connected closed manifold using a
Poincar\'e duality model of its cochain complex
\cite[Theorem~1, Propositions~1 and~2]{FelixThomas}. The common
cohomology of $K$ and $X_K$, together with its pairing,
constitutes such a model for each.
\end{remark}

\subsection{The Hochschild cohomology model}\label{sub:hochmodel}

Let $A$ be a graded commutative algebra, free of finite type
over $R$, that admits a linear form $\int\colon A\to R$ for which
$\Theta\colon A\to A^\vee$, $\Theta(a)(b)=\int ab$, is an
isomorphism. Provide $A^\vee$ with the graded bimodule structure
defined as follows:
\[
(\lambda c)(b)=\lambda(cb)
\qquad\text{and}\qquad
(c\lambda)(b)=(-1)^{|c|(|\lambda|+|b|)}\lambda(bc).
\]
Then $\Theta(ac)=\Theta(a)c$, and
\[
\Theta(ca)(b)=\int cab=(-1)^{|c|(|a|+|b|)}\int abc=(c\,\Theta(a))(b)
\]
by graded commutativity, so $\Theta$ is an isomorphism of
$A$-bimodules. The Hochschild cohomology $HH^*(A;A)$ with the cup
product and the BV operator $\Theta_*^{-1}B^\vee\Theta_*$, where
$B$ is Connes' operator on Hochschild chains, is a
Batalin-Vilkovisky algebra \cite{TradlerBV},
\cite[Section~5]{Menichi}.

\begin{theorem}[Hochschild cohomology model]\label{thm:hochmodel}
Let $(G,R)$ be an exterior pair with $G$ simply connected. Then
\[
\HH_*(LG;R)\;\cong\;HH^*\big(H^*(G;R);H^*(G;R)\big)
\]
as Batalin-Vilkovisky algebras, the right-hand side carrying the
cup product and the BV operator induced by Connes' operator through
Poincar\'e duality $H^*(G;R)\cong H^*(G;R)^\vee$.
\end{theorem}

\begin{proof}
By Theorem~\ref{thm:rigid}, $\HH_*(LG;R)\cong\HH_*(LX_G;R)$ as
BV algebras, and $H^*(G;R)\cong H^*(X_G;R)$ as graded algebras;
an algebra isomorphism changes $\Theta$ by a unit, which
cancels in $\Theta_*^{-1}B^\vee\Theta_*$. So it suffices to
treat $X_G$.

For one sphere, $A=\Lambda_R(x)$ with $|x|=n$ odd, Menichi
computes $HH^*(A;A)$ with its BV operator over $\F_2$
\cite[Proposition~20]{Menichi} from Connes' operator on the
Hochschild chains \cite[Example~12]{Menichi}; the example is
stated over any commutative ring, and for $n$ odd its signs make
the differential vanish and the BV operator integral, so
\[
HH^*(A;A)=\Lambda_R(g)\otimes R[f],\qquad
\Delta(gf^k)=k\,f^{k-1},\qquad \Delta(f^k)=0 ,
\]
which is $\HH_*(LS^n;R)$ \cite[Theorem~16]{Menichi}.

For products, the K\"unneth map
$HH^*(A;A)\otimes HH^*(B;B)\to HH^*(A\otimes B;A\otimes B)$ of
symmetric algebras is a map of BV algebras, an isomorphism when
one factor has projective Hochschild cohomology
\cite[Theorems~3.5 and~3.8, Corollary~3.9]{DuarteAngel}. Their
algebras are ungraded and their ring hereditary, but the
argument applies with Koszul signs to $\Lambda_R(x)$ over any
$R$. With Lemma~\ref{lem:product} on the loop side,
$\HH_*(LX_G;R)\cong HH^*\big(H^*(X_G;R);H^*(X_G;R)\big)$ as BV
algebras.
\end{proof}

\noindent The isomorphism of Theorem~\ref{thm:hochmodel} is
abstract. We suggest that there is a natural choice for such an
isomorphism. The
natural candidate requires a duality of the cochains, a homotopy
version of $\Theta$. To describe it, let $M$ be a closed oriented
manifold, simply connected, and $C=C^*(M;R)$. As mentioned in the
introduction, Jones' isomorphism
\[
J\colon H_*(LM;R)\longrightarrow HH^*(C;C^\vee)
\]
\cite{Jones} sends $c_*[M]$, the image of the fundamental class
under the inclusion of the constant loops $c\colon M\to LM$, to a
class $\theta:=J(c_*[M])\in HH^*(C;C^\vee)$. This class is a
morphism $\theta\colon C\to C^\vee$ in the derived category of
$C$-bimodules, and its restriction along the unit is Poincar\'e
duality. Indeed, restriction along the unit $R\to C$ is the map
$HH^*(C;C^\vee)\to H^*(C^\vee)=H_*(M;R)$, which under $J$ is the
evaluation $\mathrm{ev}_*\colon H_*(LM;R)\to H_*(M;R)$ at the
basepoint of the circle. Since $\mathrm{ev}\circ c=\mathrm{id}$,
the image of $\theta$ is $[M]$, so on cohomology $\theta$ is the
bimodule map $a\mapsto\langle a\,\cdot\,,[M]\rangle$, which is
$\Theta$. Over a field, Menichi \cite[Theorem~22]{MenichiBV} has
shown that the cap product with $\theta$ is an isomorphism
\[
HH^*(C;C)\cong HH^*(C;C^\vee),
\]
and that Connes' operator transported through it makes
$HH^*(C;C)$ a Batalin-Vilkovisky algebra. Tradler
\cite{TradlerBV} reaches the same result using an
$\infty$-inner product on $C$.

Whether the Cohen-Jones ring isomorphism
\[
\HH_*(LM;R)\cong HH^*(C;C)
\]
\cite[Theorem~3]{CohenJones} is an isomorphism of
Batalin-Vilkovisky algebras for this structure is a theorem over
a field of characteristic zero \cite{FelixThomas} and a conjecture in
general \cite{TradlerBV}, \cite[Conjecture~4]{MenichiBV}.

\begin{conjecture}[Formality with duality]\label{conj:pdformal}
Let $(G,R)$ be an exterior pair with $G$ simply connected,
$\Theta$ the Poincar\'e duality isomorphism of $H^*(G;R)$ and
$\theta$ the class of the constant loops for $C^*(G;R)$. There
is an $A_\infty$ quasi-isomorphism
$f\colon H^*(G;R)\to C^*(G;R)$ carrying $\Theta$ to $\theta$.
\end{conjecture}

\noindent Such an $f$ induces an isomorphism
$f^\#\colon HH^*(C^*(G;R);C^*(G;R))\to HH^*(H^*(G;R);H^*(G;R))$
of Gerstenhaber algebras, and since $f$ carries $\theta$ to
$\Theta$ it intertwines the two transported Connes operators:
$f^\#$ is an isomorphism of Batalin-Vilkovisky algebras.
If the Cohen-Jones isomorphism is one of Batalin-Vilkovisky
algebras, the composite
\[
\HH_*(LG;R)\;\cong\;HH^*\big(C^*(G;R);C^*(G;R)\big)
\;\xrightarrow{\ f^\#\ }\;HH^*\big(H^*(G;R);H^*(G;R)\big)
\]
is an isomorphism of Batalin-Vilkovisky algebras realizing
Theorem~\ref{thm:hochmodel}. The theorem says that the target
is the right BV algebra; the conjecture asks that the natural
map reach it.

The conjecture is true over $\Q$: $G$ is formal, and by
Remark~\ref{rem:rational} its cohomology $H^*(G;\Q)$, with its
pairing, is a Poincar\'e duality model of $C^*(G;\Q)$ in the
sense of F\'elix and Thomas \cite{FelixThomas}, for which the map
$\theta$ is strict. Thus we can find a map of algebras $f$ that
carries $\Theta$ to $\theta$. However,
formality alone is not enough. For example, the sphere $S^2$ is
formal over $\F_2$, but the duality that $H^*(S^2;\F_2)$ inherits
from the cochains has nonzero higher terms \cite{PoirierTradler}.
Therefore, with the strict $\Theta$ above, the Hochschild BV
algebra differs from $\HH_*(LS^2;\F_2)$
\cite[Corollary~30]{Menichi}. What the conjecture requires is a
formality map which at the same time is a map of the dualities.
Note that $(S^2,\F_2)$ is not an exterior pair: $S^2$ is not a
Lie group, and the generator of $H^*(S^2;\F_2)$ has even degree.
So $(S^2,\F_2)$ is not a counterexample.

\section{Torsion in the fundamental group}\label{sub:torsion}

Let $(G,R)$ be an exterior pair, and write
$G\cong T^k\times K/\Gamma$ with $\Gamma$ the torsion subgroup of
$\pi_1(G)$ (Section~\ref{sub:census}). The loop side splits,
$H_*(\loops(K/\Gamma);R)\cong R[\Gamma]\otimes H_*(\loops K;R)$
(Proposition~\ref{prop:loopsplit}). In this section we show that
the BV algebra splits the same way (Theorem~\ref{thm:torsionss}),
which extends Theorem~\ref{thm:rigid} to
$\HH_*(LG;R)\cong\HH_*(LX_G;R)\otimes R[\Gamma]$
(Theorem~\ref{thm:rigidtorsion}). We then show that over an integral domain
the graded algebra $\HH_*(LG;R)$ already determines the exponents
and $R[\Gamma]$, hence the BV algebra
(Corollary~\ref{cor:torsionclass}), and over a subring of $\Q$
also $\pi_1(G)$ (Corollary~\ref{cor:rationalpi1}).

The intersection ring of $G_{ss}$ is that of $K$
(Lemma~\ref{lem:intquot}), and so is the homology suspension on
each component of $\loops G_{ss}$ (Lemma~\ref{lem:suspcomp});
with Proposition~\ref{prop:loopsplit} this gives:

\begin{theorem}[The semisimple part]\label{thm:torsionss}
Let $K$ be a compact simply connected Lie group with $(K,R)$
exterior, $\Gamma\subset K$ a finite central subgroup with
$|\Gamma|$ a unit in $R$, and $G_{ss}=K/\Gamma$. Then
\[
\Xi:=\Phi_*\otimes\varphi\colon
\HH_*(LK;R)\otimes R[\Gamma]\;\longrightarrow\;
\HH_*(LG_{ss};R),
\qquad
(b\otimes x)\otimes g\;\longmapsto\;(r_g)_*b\otimes\varphi(x),
\]
is an isomorphism of BV algebras, with $R[\Gamma]$ in degree zero
and zero BV operator.
\end{theorem}

\begin{proof}
By Theorem~\ref{thm:hepworth1} for $K$ and for $G_{ss}$,
$\Xi$ is $\Phi_*$ on the Pontryagin factor and $\varphi$ on
the intersection factor, both ring isomorphisms
(Proposition~\ref{prop:loopsplit}, Lemma~\ref{lem:intquot}),
and $R[\Gamma]$ is in degree zero. So $\Xi$ is a ring
isomorphism.

For the BV operators let $b\in H_*(\loops K;R)$ with
$\psi(b)=\sum b_{(1)}\otimes b_{(2)}$. The isomorphism $\Phi_*$ is
a map of coalgebras and $g$ is group-like, so
$\psi((r_g)_*b)=\sum(r_g)_*b_{(1)}\otimes(r_g)_*b_{(2)}$, and
Theorem~\ref{thm:hepworth1} for $G_{ss}$ gives
\[
\Delta_{G_{ss}}\Xi\big((b\otimes x)\otimes g\big)
=\sum(r_g)_*b_{(1)}\otimes
\sigma_{G_{ss}}\big((r_g)_*b_{(2)}\big)\,\varphi(x).
\]
By Lemma~\ref{lem:suspcomp},
$\sigma_{G_{ss}}((r_g)_*b_{(2)})$ is $0$ for $b_{(2)}=1$ and
$p_*\sigma_K(b_{(2)})$ otherwise, and by
Lemma~\ref{lem:intquot},
$p_*\sigma_K(b_{(2)})\,\varphi(x)=\varphi(\sigma_K(b_{(2)})\,x)$.
Hence
\[
\Delta_{G_{ss}}\Xi\big((b\otimes x)\otimes g\big)
=\sum(r_g)_*b_{(1)}\otimes\varphi\big(\sigma_K(b_{(2)})\,x\big)
=\Xi\big(\Delta_K(b\otimes x)\otimes g\big),
\]
the last equality being Theorem~\ref{thm:hepworth1} for $K$,
with $\sigma_K(1)=0$.
\end{proof}

\begin{theorem}[BV rigidity for all exterior pairs]\label{thm:rigidtorsion}
Let $(G,R)$ be an exterior pair and $\Gamma$ the torsion
subgroup of $\pi_1(G)$. Then there is an isomorphism of BV
algebras
\[
\HH_*(LG;R)\;\cong\;\HH_*(LX_G;R)\otimes R[\Gamma],
\]
with $R[\Gamma]$ the group ring of $\Gamma$, placed in degree
zero with zero BV operator, so that
$\Delta_G=\Delta_{X_G}\otimes1$. In particular
$\HH_*(LG;\Q)\cong\HH_*(LX_G;\Q)\otimes\Q[\Gamma]$ for every
connected compact Lie group $G$.
\end{theorem}

\begin{proof}
The exponents of $G$ are those of $K$ together with $k$
zeros, so $X_G=T^k\times X_K$. By Lemma~\ref{lem:product},
Theorem~\ref{thm:torsionss} and Theorem~\ref{thm:rigid} for
$(K,R)$,
\begin{align*}
\HH_*(LG;R)
&\cong\HH_*(LT^k;R)\otimes\HH_*(LK;R)\otimes R[\Gamma]\\
&\cong\HH_*(LT^k;R)\otimes\HH_*(LX_K;R)\otimes R[\Gamma]
\cong\HH_*(LX_G;R)\otimes R[\Gamma]
\end{align*}
as BV algebras. Over $\Q$ every
pair is exterior (Theorem~\ref{thm:whichpairs}).
\end{proof}

\begin{corollary}[Classification over a domain]\label{cor:torsionclass}
Let $R$ be a domain, and let $(G,R)$, $(G',R)$ be exterior pairs
with torsion $\Gamma$, $\Gamma'$. Then the following three
statements are equivalent:
\begin{enumerate}
\item $\HH_*(LG;R)\cong\HH_*(LG';R)$ as BV algebras;
\item $\HH_*(LG;R)\cong\HH_*(LG';R)$ as graded $R$-algebras;
\item $G$ and $G'$ have the same $R$-exponents, and
$R[\Gamma]\cong R[\Gamma']$ as $R$-algebras.
\end{enumerate}
\end{corollary}

\begin{proof}
(3)$\Rightarrow$(1) follows from Theorem~\ref{thm:rigidtorsion}, while
(1)$\Rightarrow$(2) is immediate. To prove (2)$\Rightarrow$(3),
let $k$ be the number of times the exponent $0$ appears in the
list of the $r$ exponents of $G$, and let
$B=R[\Gamma][t_1^{\pm1},\dots,t_k^{\pm1}]$. Theorem~\ref{thm:rigidtorsion}
combined with Proposition~\ref{prop:spheres} then gives
\[
A:=\HH_*(LG;R)\;\cong\;B[u_{k+1},\dots,u_r]\otimes
\Lambda_R(\theta_1,\dots,\theta_k,d_{k+1},\dots,d_r)
\]
as graded algebras, where $|u_j|=2m_j$. Let $F$ denote the
quotient field of $R$. Since $|\Gamma|$ is a unit in $F$, we have
$F[\Gamma]\cong\prod_iL_i$, where each $L_i$ is a finite
extension of $F$. Therefore
$B\otimes_RF\cong\prod_iL_i[t_1^{\pm1},\dots,t_k^{\pm1}]$ is
reduced with Krull dimension $k$, and $B$, being contained in
$B\otimes_RF$, is also reduced. Hence the odd generators generate
the nilradical $N$ of $A$, and $A/N\cong B[u_{k+1},\dots,u_r]$.
Therefore an isomorphism of graded $R$-algebras $A\cong A'$
induces isomorphisms of the degree-zero components $B\cong B'$ of
the reduced rings, of the extended rings $B\otimes_RF$, and of
the indecomposables of $A/N$ over $B$, which are free on the
$u_j$. Thus we see that $k$ is the Krull dimension of
$B\otimes_RF$, the positive exponents are the degrees of the
$u_j$, and $R[\Gamma]$ is the integral closure of $R$ in $B$:
$R[\Gamma]$ is finitely generated over $R$, and an element of $B$
that is integral over $R$ is algebraic over $F$ in
$\prod_iL_i[t^{\pm1}]$, thus constant, and thus lies in
$B\cap F[\Gamma]=R[\Gamma]$.
\end{proof}

\begin{remark}[Sharpness of Theorem~\ref{thm:rigid}]
Let $R$ be a domain and $\Gamma\neq1$. Then $\HH_*(LG;R)$ and
$\HH_*(LX_G;R)$ are not isomorphic even as rings. The
degree-zero parts of their reduced rings are
$R[\Gamma][t_1^{\pm1},\dots,t_k^{\pm1}]$ and
$R[t_1^{\pm1},\dots,t_k^{\pm1}]$, and the first has zero
divisors: $(g-1)(1+g+\dots+g^{n-1})=0$ for $g$ of order $n>1$.
\end{remark}

For finite abelian $\Gamma$, $\Gamma'$, Perlis and Walker
show that $\Q[\Gamma]\cong\Q[\Gamma']$ only if
$\Gamma\cong\Gamma'$ \cite[Theorem~3 with Corollary~1]{PerlisWalker}.
For a subring $R\subseteq\Q$ the same follows on tensoring
$R[\Gamma]\cong R[\Gamma']$ with $\Q$. As
$\pi_1(G)\cong\Z^k\times\Gamma$ with $k$ the multiplicity of
the exponent $0$, Corollary~\ref{cor:torsionclass} becomes:

\begin{corollary}[The BV algebra determines $\pi_1$]\label{cor:rationalpi1}
Let $R$ be a subring of $\Q$, for instance $\Q$, $\Z[1/n]$ or
$\Z_{(p)}$, and $(G,R)$, $(G',R)$ exterior pairs. The following
are equivalent:
\begin{enumerate}
\item $\HH_*(LG;R)\cong\HH_*(LG';R)$ as BV algebras;
\item $\HH_*(LG;R)\cong\HH_*(LG';R)$ as graded algebras;
\item $G$ and $G'$ have the same $R$-exponents and
$\pi_1(G)\cong\pi_1(G')$.
\end{enumerate}
In particular this holds over $\Q$ for all connected compact
Lie groups $G$, $G'$.
\end{corollary}

\noindent The covering $K\to G_{ss}$ is an isomorphism on
rational homology (Proposition~\ref{prop:covering}) and on
$\pi_n$ for $n\geq2$, and $\pi_1(G)\otimes\Q=\Q^k$: the finite
group $\Gamma$ is invisible to the rational invariants of the
space $G$. It is not invisible to the rational BV algebra of
$LG$, where it sits in $H_0(LG;\Q)=\Q[\pi_1(G)]$.

\begin{remark}[Over $\F_p$]
The BV algebra sees $\pi_1(G)$ only through the group ring
$R[\pi_1(G)]$, and in the case of $\F_p$ this carries less
information than the group itself. In fact,
$\F_p[\Z/4]\cong\F_p^4\cong\F_p[\Z/2\times\Z/2]$ when
$p\equiv1 \pmod 4$, so there exist two exterior pairs over $\F_p$
with the same exponents and non-isomorphic torsion whose BV
algebras are isomorphic. The question does not occur over $\Z$,
because then $\Gamma=1$.
\end{remark}

\begin{example}[Which groups]\label{ex:torsioncensus}
As shown by Corollary~\ref{thm:borel}, $(G,R)$ is an exterior
pair if and only if $|\Gamma|$ and the torsion primes of the
simple factors of $K$ (Proposition~\ref{prop:census}) are units
in $R$. This obviously holds when $R=\Q$. The groups with
$\Gamma\neq1$ are the quotients of $K$ by nontrivial subgroups
of its centre. When $K$ is simple, the centres are listed in
\cite[\S0.4]{Borel61}, and $(G_{ss},R)$ is an exterior pair if
and only if the primes in the rightmost column are units in
$R$:
\begin{center}
\begin{tabular}{lll}
$G_{ss}$ & $\Gamma$ & units in $R$\\[2pt]
\hline
$\SU(n)/(\Z/d)$, $d\mid n$, $d>1$; $d=n$ is $\mathrm{PU}(n)$ & $\Z/d$ & primes dividing $d$\\
$\mathrm{PSp}(n)=\Sp(n)/(\Z/2)$ & $\Z/2$ & $2$\\
$\mathrm{SO}(n)$, $n\geq3$ & $\Z/2$ & $2$\\
$\mathrm{Spin}(4m{+}2)/(\Z/4)$ & $\Z/4$ & $2$\\
$\mathrm{Spin}(4m)/H$, $1\neq H\subset(\Z/2)^2$ & $H$ & $2$\\
$E_6/(\Z/3)$ & $\Z/3$ & $2,3$\\
$E_7/(\Z/2)$ & $\Z/2$ & $2,3$
\end{tabular}
\end{center}
For $K=K_1\times\cdots\times K_s$, $\Gamma$ is any subgroup of
$\prod_iZ(K_i)$, as for $\mathrm{SO}(4)=(\SU(2)\times\SU(2))/(\Z/2)$.
A central quotient of $T^k\times K$ need not have torsion in
$\pi_1$: $\U(n)=(S^1\times\SU(n))/(\Z/n)$ has $\pi_1=\Z$.
\end{example}

\begin{remark}[Hepworth's $\mathrm{SO}(n)$]\label{rem:hepworthSO}
Hepworth proves that \cite[Theorem~4]{Hepworth}
\[
\HH_*(L\,\mathrm{SO}(2m{+}1);\Q)
=\Q[\alpha_0,\dots,\alpha_{2m-1}]/U_m\otimes
\Lambda_\Q[\beta_3,\dots,\beta_{4m-1}],
\qquad
\Delta=\sum_{i=1}^m\partial_i\otimes\delta_i,
\]
where
\[
\partial_i(\alpha_j)=
\begin{cases}
\alpha_{j-2i+1}, & j\geq 2i-1,\\
0, & \text{otherwise.}
\end{cases}
\]
In the proof of his Proposition~18, he writes
$\alpha_j=\omega\otimes\sigma_j$, with $\omega$ the generator of
$\pi_1=\Z/2$ and $\sigma_j$ Bott's generators of
$H_*(\loops_0\mathrm{SO}(2m{+}1);\Q)$. Also note that $U_m$
consists of the relation $\omega^2=1$ together with Bott's
relations times $\omega^2$. Since $\partial_i(\alpha_0)=0$, each
$\partial_i$ acts only on the $\sigma_j$. The BV operator is thus
equal to
\[
1\otimes\Delta_{\mathrm{Spin}(2m+1)},
\]
as predicted by Theorem~\ref{thm:torsionss}, while the shear of
Proposition~\ref{thm:shear} makes it a pure tensor over the
product of spheres $S^3\times\cdots\times S^{4m-1}$. Similarly,
for $\mathrm{SO}(2m{+}2)$ there is one additional generator
$\varepsilon_m=1\otimes\varepsilon$, carrying no $\omega$.
\end{remark}

\begin{remark}[Torsion primes]
For $p\mid|\Gamma|$ the pair $(G,\F_p)$ is not an exterior pair
(Remark~\ref{rem:pi1tors}): then $p_*$ is not an isomorphism.
Furthermore, in Lemma~\ref{lem:suspcomp} there exists $g$ in
$H_1(G_{ss};\F_p)=\Gamma\otimes\F_p$ whose Hurewicz image is
nonzero; therefore the BV operator acts on the group ring, as in
Hepworth's computations modulo $2$ for $\mathrm{SO}(n)$
\cite{Hepworth}.
\end{remark}

\bibliographystyle{amsplain}
\bibliography{lspbv}

\end{document}